\documentclass{zml-article}

\usepackage{amsmath}
\usepackage{amssymb}
\usepackage{wasysym}
\usepackage{url}
\usepackage{amsthm}
\usepackage{graphicx}
\usepackage{amscd}
\usepackage{stmaryrd}
\usepackage[usenames,dvipsnames,svgnames,table]{xcolor}
\usepackage{todonotes}
\usepackage{cleveref}
\usepackage{mdframed}
\usepackage{extarrows}
\usepackage{mathtools}
\usepackage{thm-restate}
\usepackage{fontenc}
\usepackage{quiver}
\usepackage{tikz}
\usetikzlibrary{decorations.pathmorphing}

\ifdefined\cA
\else
\newcommand\cA{\mathcal{A}}
\fi
\ifdefined\cB
\else
\newcommand\cB{\mathcal{B}}
\fi
\ifdefined\cC
\else
\newcommand\cC{\mathcal{C}}
\fi
\ifdefined\cD
\else
\newcommand\cD{\mathcal{D}}
\fi
\ifdefined\cE
\else
\newcommand\cE{\mathcal{E}}
\fi
\ifdefined\cF
\else
\newcommand\cF{\mathcal{F}}
\fi
\ifdefined\cG
\else
\newcommand\cG{\mathcal{G}}
\fi
\ifdefined\cH
\else
\newcommand\cH{\mathcal{H}}
\fi
\ifdefined\cI
\else
\newcommand\cI{\mathcal{I}}
\fi
\ifdefined\cJ
\else
\newcommand\cJ{\mathcal{J}}
\fi
\ifdefined\cK
\else
\newcommand\cK{\mathcal{K}}
\fi
\ifdefined\cL
\else
\newcommand\cL{\mathcal{L}}
\fi
\ifdefined\cM
\else
\newcommand\cM{\mathcal{M}}
\fi
\ifdefined\cN
\else
\newcommand\cN{\mathcal{N}}
\fi
\ifdefined\cO
\else
\newcommand\cO{\mathcal{O}}
\fi
\ifdefined\cP
\else
\newcommand\cP{\mathcal{P}}
\fi
\ifdefined\cQ
\else
\newcommand\cQ{\mathcal{Q}}
\fi
\ifdefined\cR
\else
\newcommand\cR{\mathcal{R}}
\fi
\ifdefined\cS
\else
\newcommand\cS{\mathcal{S}}
\fi
\ifdefined\cT
\else
\newcommand\cT{\mathcal{T}}
\fi
\ifdefined\cU
\else
\newcommand\cU{\mathcal{U}}
\fi
\ifdefined\cV
\else
\newcommand\cV{\mathcal{V}}
\fi
\ifdefined\cW
\else
\newcommand\cW{\mathcal{W}}
\fi
\ifdefined\cX
\else
\newcommand\cX{\mathcal{X}}
\fi
\ifdefined\cY
\else
\newcommand\cY{\mathcal{Y}}
\fi
\ifdefined\cZ
\else
\newcommand\cZ{\mathcal{Z}}
\fi
\ifdefined\bA
\else
\newcommand\bA{\mathbb{A}}
\fi
\ifdefined\bB
\else
\newcommand\bB{\mathbb{B}}
\fi
\ifdefined\bC
\else
\newcommand\bC{\mathbb{C}}
\fi
\ifdefined\bD
\else
\newcommand\bD{\mathbb{D}}
\fi
\ifdefined\bE
\else
\newcommand\bE{\mathbb{E}}
\fi
\ifdefined\bF
\else
\newcommand\bF{\mathbb{F}}
\fi
\ifdefined\bG
\else
\newcommand\bG{\mathbb{G}}
\fi
\ifdefined\bH
\else
\newcommand\bH{\mathbb{H}}
\fi
\ifdefined\bI
\else
\newcommand\bI{\mathbb{I}}
\fi
\ifdefined\bJ
\else
\newcommand\bJ{\mathbb{J}}
\fi
\ifdefined\bK
\else
\newcommand\bK{\mathbb{K}}
\fi
\ifdefined\bL
\else
\newcommand\bL{\mathbb{L}}
\fi
\ifdefined\bM
\else
\newcommand\bM{\mathbb{M}}
\fi
\ifdefined\bN
\else
\newcommand\bN{\mathbb{N}}
\fi
\ifdefined\bO
\else
\newcommand\bO{\mathbb{O}}
\fi
\ifdefined\bP
\else
\newcommand\bP{\mathbb{P}}
\fi
\ifdefined\bQ
\else
\newcommand\bQ{\mathbb{Q}}
\fi
\ifdefined\bR
\else
\newcommand\bR{\mathbb{R}}
\fi
\ifdefined\bS
\else
\newcommand\bS{\mathbb{S}}
\fi
\ifdefined\bT
\else
\newcommand\bT{\mathbb{T}}
\fi
\ifdefined\bU
\else
\newcommand\bU{\mathbb{U}}
\fi
\ifdefined\bV
\else
\newcommand\bV{\mathbb{V}}
\fi
\ifdefined\bW
\else
\newcommand\bW{\mathbb{W}}
\fi
\ifdefined\bX
\else
\newcommand\bX{\mathbb{X}}
\fi
\ifdefined\bY
\else
\newcommand\bY{\mathbb{Y}}
\fi
\ifdefined\bZ
\else
\newcommand\bZ{\mathbb{Z}}
\fi

\theoremstyle{definition}
\newtheorem{theorem}{Theorem}
\newtheorem{remark}[theorem]{Remark}
\newtheorem{definition}[theorem]{Definition}

\newtheorem{corollary}[theorem]{Corollary}
\newtheorem{proposition}[theorem]{Proposition}
\newtheorem{example}[theorem]{Example}
\newtheorem{lemma}[theorem]{Lemma}
\newtheorem{claim}[theorem]{Claim}

\newcommand{\Cantor}{{2^\mathbb{N}}}

\newcommand{\id}{\textrm{id}}
\newcommand{\dom}{\operatorname{dom}}

\newcommand{\Baire}{{\mathbb{N}^\mathbb{N}}}
\newcommand{\hide}[1]{}

\newcommand{\partto}{\rightharpoonup}

\newcommand{\Ninfty}{\mathbb{N}_\infty}
\newcommand\powerset{\mathcal{P}}
\newcommand{\inl}{{\sf inl}}
\newcommand{\inr}{{\sf inr}}
\newcommand\LPO{{\sf LPO}}
\newcommand\LLPO{{\sf LLPO}}
\newcommand\dneg{\neg\neg}
\newcommand\modal{\bigcirc}
\newcommand\longto{\longrightarrow}
\newcommand\MP{{\sf MP}}
\newcommand\Succ{\mathsf{S}}
\newcommand\Succinfty{\underline{\mathsf{S}}}
\newcommand\tuple[1]{\left\langle #1 \right\rangle}

\newlength{\LETTERheight}
\AtBeginDocument{\settoheight{\LETTERheight}{I}}
\newcommand*{\squiggle}{\ \raisebox{0.24\LETTERheight}{\tikz \draw [-,
line join=round,
decorate, decoration={
    zigzag,
    segment length=4,
    amplitude=.9,
    post length=0pt,
    pre length=0pt
}] (0,0) -- (0.4,0);}\ }

\title{The Myhill isomorphism theorem does not generalize much}

\author[1]{Cécilia Pradic}
 
\affil[1]{Unemployed}

\corrauthor{Cécilia Pradic}{cecilia.pradic@ens-lyon.org}

\thanks{%
Thank you to Eike Neumann, Arno Pauly and Manlio Valenti for substantial discussions
on this work and to Andrej Bauer for asking a motivating question and encouragements.
I also also thank the anonymous reviewer for their comments which led
to improving the paper.
}

\abstract{%
The Myhill isomorphism theorem is a variant of the Cantor-Bernstein
theorem. It states that, from two injections that reduce two
subsets of $\bN$ to one another, there exists a bijection $\bN \to \bN$
that preserves them. 
This theorem can be proven constructively.

We investigate to which extent the theorem can be extended to infinite
sets other than $\bN$. We show that, assuming Markov's principle, the theorem
can be extended to the conatural numbers $\Ninfty$ provided that we only require
that bicomplemented sets are preserved by the bijection. This restriction is
essential. Otherwise, the picture is overall negative: 
among other things, it is impossible to extend that result to either
$2 \times \Ninfty$, $\bN + \Ninfty$, $\bN \times \Ninfty$, $\Ninfty^2$,
$\Cantor$ or $\Baire$.
}

\begin{document}

\maketitle
The Myhill isomorphism theorem is a constructive refinement of the Cantor-Bernstein
theorem which states the following: for any two sets $A, B \subseteq \bN$
and injections $f, g : \bN \to \bN$ with $A = f^{-1}(B)$ and $B = g^{-1}(A)$,
there is a bijection $h : \bN \to \bN$ that restricts to a bijection between $A$ and $B$.
Since it is provable constructively, this statement remains true if we further
require that $h$ is uniformly computable from $f$ and $g$.
Typical proofs of Cantor-Bernstein exhibit non-constructive but explicit bijections $h$ that
would also preserve $A$ and $B$; hence the whole point is that this theorem
can be proven effectively via a more sophisticated back-and-forth argument.
Andrej Bauer asked to which extent this theorem can be constructively extended to sets other than
$\mathbb{N}$.
To be precise, he proposed the following definition\footnote{See \url{https://mathstodon.xyz/@andrejbauer/114427716438955347}.}.

\begin{definition}
Given a set $X$, subsets $A, B \subseteq X$, say that $A \preceq^X B$ ($A$ \emph{reduces to} $B$)
if there is a function $f : X \to X$ such that $A = f^{-1}(B)$.
Say that $A \preceq_1^X B$ ($A$ \emph{injectively} reduces to $B$)
if $A \preceq^X B$ via an injective map $f$ and that $A \simeq^X B$ ($A$ \emph{bijectively} reduces to $B$)
if we have a bijection $f : X \to X$ that restricts to a bijection between $A$ and $B$.

We say that $X$ has the \emph{Myhill property} if for every $A, B \subseteq X$,
$A \preceq_1^X B$ and $B \preceq_1^X A$ imply $A \simeq^X B$.
\end{definition}

The Myhill isomorphism theorem states that $\mathbb{N}$ has the Myhill property,
and this result can be extended to discrete enumerable sets using the same
proof idea~\cite[Theorem 3.3]{FJS23}. Bauer specifically asked:
\begin{enumerate}
  \item \label{enumitem:ninftymyhill} does the set of conatural numbers $\mathbb{N}_\infty$ have the
Myhill property?
\item \label{enumitem:myhillstable} if $X$ and $Y$ have the Myhill property, is it necessarily the case of 
$X + Y$, $X \times Y$ and $X^Y$?
\end{enumerate}

We mostly answer these questions negatively. In~\Cref{sec:negative},
we show that in Kleene-Vesley realizability, $\Ninfty$ does \emph{not} have
the Myhill property.
We then show that stability under $X + Y$ and $X \times Y$ of the Myhill property
all entail that
excluded middle can be lifted from $\mathbf{\Sigma^0_1}$ formulas to all sentences.
In~\Cref{sec:positive}, we attempt to salvage the situation for $\Ninfty$
somewhat: we show that if we relax our expectations so that we only require
that the bijection be a reduction between the bicomplements of $A$ and $B$, we
can build it assuming only Markov's principle. The proof strategy is essentially
an elaboration of the one for $\bN$ and allows to show that $\Ninfty$ has
what we call the $\neg\neg$ strong Myhill property.
Finally, in~\Cref{sec:negnegative}, we show that this notion
is not much better behaved by exhibiting some nice sets that do not have the
$\neg\neg$ strong Myhill property and showing that stability of the $\dneg$ Myhill
property under $\times, +$
and exponentiation turns Markov's principle into excluded middle.

\subparagraph*{Foundational and notational preliminaries}
We work informally in intuitionistic Zermelo set theory~\cite{myhill1973some}
where we restrict comprehension to bounded comprehension and we do not assume
foundation. Minute stylistic issues aside, we could also work in the internal logic
of toposes~\cite[Chapter VI]{MacLaneMoerdijk}.
Among other things, this means that our universe of sets is closed under pairing, union, powerset
(which we write $\powerset(-)$), set comprehension $\{ x \in A \mid \varphi \}$
for bounded formulas $\varphi$ and contains the set of natural numbers $\bN$.
For sets $A$ and $B$, function spaces $B^A$, cartesian products $A \times B$
and disjoint unions $A + B$ are built as usual.
We write $\inl:A\to A+B$ and $\inr:B\to A+B$ for the usual injections into disjoint unions,
and given $f : A \to Z$ and $g : B \to Z$, we write $[f,g]$ for the copairing
$A + B \to Z$ such that $[f,g](\inl(a)) = f(a)$ and $[f,g](\inr(b)) = g(b)$.
In particular, for every $x\in A+B$ there is either an $a\in A$
such that $x=\inl(a)$ or $b\in B$ such that $x=\inr(b)$.
To make notations lighter by avoiding an overuse of the coprojections $\inl$
and $\inr$, we may work with disjoint unions $A \cup B$ instead of
the isomorphic $A + B$
when reasonable. We then write $A \uplus B$ to stress the union of $A$ and $B$
is disjoint; for convenience, we will assume $\bN$ and $\Cantor$ to be disjoint.
We write $\pi_i$ for the projection $A_1 \times A_2 \to A_i$ for $i = 0, 1$.

We regard natural numbers as ordinals throughout (i.e.,  $n = \{0, \ldots, n - 1\}$)
and sometimes write $\Succ : \bN \to \bN$ for the map $n \mapsto n + 1$.
If $x \in X$, write $x^\omega$ for the constant sequence equal to $x$ of $X^\bN$,
and if $p \in X^\bN$, we may write $px$ for the sequence $q$ with $q_0 = p$ and
$q_{n+1} = x_n$.
We will write $A \partto B$ for the set of partial functions from $A$ to $B$.

In very few occasions\footnote{To be more precise, in~\Cref{lem:transfermodestcountable}, \Cref{cor:KVnotNinftyMP}, \Cref{rem:modest} and \Cref{cor:dneg2NinftyMP}.},
we will offer meta-theoretic theorems and proofs on the validity
of certain statements that rely on realizability models, especially Kleene-Vesley realizability.
Those parts are self-contained and not required to read the rest of the text, so we don't spend much
ink introducing realizability. The interested reader can refer to~\cite{BauerPhD,VanOosten} for an introduction
of the relevant notions that includes the construction of realizability toposes, including the Kleene-Vesley topos
$\mathbf{RT}(\mathcal{K}_2^{\mathrm{rec}}, \mathcal{K}_2)$
and the definitions of partitioned and modest sets therein.
In brief, $\mathbf{RT}(\mathcal{K}_2^{\mathrm{rec}}, \mathcal{K}_2)$
is the realizability model corresponding to type-2 computability (or computable analysis).
The partial combinatory algebra (PCA) of potential realizers $\mathcal{K}_2$, represents
all the classical partial continuous functions $\Cantor \partto \Cantor$ by way of trees encoded
in $\Baire$, and the subPCA $\mathcal{K}_2^{\mathrm{rec}}$ restricts $\mathcal{K}_2$
to computable elements. Then realizers for theorems are espected to live in
$\mathcal{K}_2^{\mathrm{rec}}$, while their correctness is defined with respect to
interactions with potentially other elements of $\mathcal{K}_2$.

\section{Background}

\subsection{The conatural numbers}

The set of conatural numbers $\Ninfty$ can be characterized up to isomorphism
as a final coalgebra for the functor $X \mapsto 1 + X$ (while, dually, $\bN$ is
an initial algebra for it).
It is also 
sometimes referred to as the one-point-compactification of $\mathbb{N}$,
completing $\mathbb{N}$ with a ``point at infinity''.
We recall here some of its relevant
properties; we omit the easier proofs and refer the reader to~\cite{Esc13,type-topology} for
more systematic expositions.

We officially define $\Ninfty \subseteq 2^{\mathbb{N}}$ as the set of bit sequences that
contain at most a single $1$. The natural number injects into $\Ninfty$ 
by taking $n$ to the sequence $0^n 1 0^\omega \in 2^\bN$; write
that injection $n \mapsto \underline{n}$ and its image $\underline{\bN}$.
The point at infinity is the constant $0$ sequence $0^\omega$ that we write
$\infty$ in the sequel.
Classically, $\Ninfty$ is the disjoint union of $\underline{\bN}$
and $\{\infty\}$ and thus isomorphic to $\bN$.
This is not constructively valid: $\bN \cong \Ninfty$ is equivalent to
the limited principle of omniscience, i.e. excluded middle restricted to
$\mathbf{\Sigma^0_1}$ formulas.

\begin{definition}[Limited principle of omniscience ($\LPO$)]
 \label{def:lpo}
For any $p \in \Cantor$, either $p = 0^\omega$ or $\exists n \in \bN. ~ p_n = 1$.
\end{definition}

\begin{lemma}
  \label{lem:NinftyLPO}
$\LPO$ is equivalent to
$\forall p \in \Ninfty. \; p = \infty \vee  p \in \underline{\bN}$.
\end{lemma}
\begin{proof}
The direct implication is trivial by unfolding the definitions.
The converse is easily proven by considering the canonical map $2^\bN \to \Ninfty$ that suppresses
all further occurrences of $1$ after the first.
\end{proof}

An important fact is that $\Ninfty$ admits a choice operator $\varepsilon$ for decidable predicates.

\begin{theorem}[{\hspace{-0.05em}\cite[Theorem 3.15]{Esc13}}]\label{thm:Ninftysel}
  There is a function $\varepsilon : 2^{\Ninfty} \to\Ninfty$
  such that for every $Q\in 2^{\Ninfty}$,
  if $Q(\varepsilon(Q)) = 1$, then
  $\forall p\in \Ninfty. ~ Q(p) = 1$.

  In particular, for any decidable predicate $P$ over $\Ninfty$, we can
  decide whether $\exists x \in \Ninfty. \; P(x)$ or $\forall x \in \Ninfty. ~ \neg P(x)$
\end{theorem}

That said, it is important to note that $\Ninfty$ is not constructively discrete, that is,
equality between two $x, y \in \Ninfty$ is not computably decidable: the issue is that
as long as we only see zeroes in a sequence, we have no reason to stop looking
for a one. 

At times, we will also want to use that.
$\Ninfty \setminus \{\infty\} = \underline{\bN}$.
Only the right-to-left inclusion can be shown to
hold.
The left-to-right inclusion holding is equivalent to \emph{Markov's principle}.

\begin{definition}[Markov's principle ($\MP$)]
  \label{def:mp}
For every $p \in 2^\bN$, if
$p \neq 0^\omega$, then there exists $n \in \mathbb{N}$ such that $p_n = 1$.
\end{definition}

\begin{lemma}
\label{lem:NinftyMP}
The following are equivalent:
\begin{itemize}
\item Markov's principle
\item $\forall x \in \Ninfty. \; x \neq \infty \Longrightarrow \exists n \in \bN. \; x = \underline{n}$
\item $\Ninfty \setminus \{\infty\} = \underline{\bN}$
\end{itemize}
\end{lemma}

While $\MP$ cannot be proven from intuitionistic set theory, unlike other
classical axioms we discuss in this paper, it is true in
Kleene(-Vesley) realizability. But it does not hold in topological models inspired by
Brouwer's intuitionism or Kreisel's modified realizability~\cite{kreisel59mr}.

Without further axioms, one cannot construct functions $\Ninfty \to \Ninfty$
which are discontinuous. In fact, the existence of explicit witnesses of
discontinuity for some function even allows to compute $\LPO$.

\begin{definition}
Say that a point $x \in X$ is \emph{isolated} in $X$ if $\forall y \in X. \; x = y \vee x \neq y$.
\end{definition}

\begin{lemma}
  \label{lem:inftyToIsolated}
If we have an injective map $f : \Ninfty \to X$ which maps $\infty$ to an
isolated point, then $\LPO$ holds.
\end{lemma}

Noting that every $x \in \underline{\bN}$ is isolated in $\Ninfty$, we can show
the following.

\begin{lemma}
  \label{lem:badinjectionLPO}
Assuming $\MP$, if $f : \Ninfty \to \Ninfty$ is injective such
that $f(\infty) \neq \infty$, then $\LPO$ holds.
\end{lemma}
\begin{proof}
By Markov's principle, there is $n \in \bN$ such that $f(\infty) = \underline{n}$.
Since $\underline{n}$ is isolated in $\Ninfty$, we can apply~\Cref{lem:inftyToIsolated}.
\end{proof}

There is a natural order on $\Ninfty$: for $x, y \in \Ninfty$, we say that
$x \le y$ if, for every $n \in \bN$ we have
$\left(\forall k \le n. \; y_k = 0\right) \Rightarrow x_n = 0$. The canonical
map $\bN \to \Ninfty$ is a monotone embedding and $\infty$ is the maximal element.
We also have the following characterization of the order.

\begin{lemma}
  \label{lem:NinftyOrdYon}
For $x, y \in \Ninfty$, $x \le y$ is equivalent to
$\forall n \in \bN. \; \underline{n} \le x \Rightarrow \underline{n} \le y$.
\end{lemma}

We may also compute binary infima $\inf : \Ninfty \times \Ninfty \to \Ninfty$
(as well as suprema), which equips $\Ninfty$ with a lattice structure. However,
we cannot constructively show that the order $\le$ is total over $\Ninfty$.

\begin{definition}[Lesser limited principle of omniscience $(\LLPO)$]
\[\forall p \in \Ninfty. \; (\forall i \in \bN. \; p_{2i} = 0) \vee (\forall i \in \bN. \; p_{2i+1} = 0)\]
\end{definition}
\begin{proposition}
Totality of $\le$ over $\Ninfty$, i.e. $\forall x \; y \in \Ninfty. \; x \le y \vee y \le x$, is
equivalent to $\LLPO$.
\end{proposition}

In particular, we cannot in general show that $\inf(x,y) = x \vee \inf(x,y) = y$.
However, assuming Markov's principle, if we have the guarantee that $x \neq y$,
we can say more.

\begin{lemma}
  \label{lem:MPmin}
Assume $\MP$.
For any $x, y \in \Ninfty$ such that $x \neq y$ or $x \neq \infty$, $\inf(x,y) = x$ or $\inf(x,y) = y$.
Furthermore, $\inf(x,y) \in \underline{\bN}$.
\end{lemma}

If we further
assume that either $x$ or $y$ is in $\underline{\bN}$,
Markov's principle is not needed.

\begin{lemma}
  \label{lem:minwoMP}
For any $x \in \Ninfty$ and $n \in \bN$, we have either that
\begin{itemize}
 \item $x < \underline{n}$ and $x \in \underline{\bN}$
 \item or $x = \underline{n}$
 \item or $x > \underline{n}$.
\end{itemize}
\end{lemma}

We call $\Succinfty : \Ninfty \to \Ninfty $ the map $x \to 0x$ and note that
$\Succinfty(x) = x$ is equivalent to $x = \infty$. Let us also state how
$[\underline{0}, \Succinfty] : 1 + \Ninfty \to \Ninfty$ has an inverse which
equips $\Ninfty$ with a terminal coalgebra structure.

\begin{lemma}
  \label{lem:NinftyTerminalCoAlg}
For any $X$ and map $c : X \to 1 + X$, there is a unique $f : X \to \Ninfty$
such that
\[
\begin{array}{lcl!\quad c!\quad lcl}
  c(x) &=& \inl(0) &\Longleftrightarrow& f(x) &=& \underline{0} \\
  c(x) &=& \inr(x') &\Longleftrightarrow& f(x) &=& \Succinfty(f(x')) \\
\end{array}
\]
\end{lemma}
\begin{proof}
  This is standard~\cite[\texttt{Conaturals.UniversalProperty}]{type-topology}.
The basic idea is that
$f(x)_n$ is determined by $(\hat{c}^k(x))_{k \le n}$, where
$\hat{c} : 1 + X \to 1 + X$ the obvious lifting of $c$
(that is, with $\hat{c}(\inl(0)) = \inl(0)$ and $\hat{c}(\inr(x)) = c(x)$).
\end{proof}

Thanks to that, one can define an addition and multiplications on $\Ninfty$ such
that $n \mapsto \underline{n}$ is a morphism of semiring from $\bN$
to $\Ninfty$ equipped with the obvious structure. We may also define a subtraction
operation $\Ninfty \times \bN \to \Ninfty$ by recursion on the second argument
such that $(x + \underline{n}) - n = x$.

We conclude this subsection with a lemma that allows to lift some endomaps
over $\bN$ to $\Ninfty$. This lifting was already defined
in~\cite[Lemma 5.4]{Esc13}, but we will additionally use that it preserves
strictly increasing maps.

\begin{lemma}
  \label{lem:liftInflationary}
  Given an inflationary map $f : \bN \to \bN$ (i.e. such that $f(n) \ge n$
  for all $n \in \bN$),
  there is a unique monotone map $\underline{f} : \Ninfty \to \Ninfty$
  such that $\underline{f}(\underline{n}) = \underline{f(n)}$ for all
  $n \in \bN$ and $\underline{f}(\infty) = \infty$.
  Furthermore, if $f$ is strictly increasing, $\underline{f}$ is an embedding
  (and thus a strictly increasing injection).
\end{lemma}
\begin{proof}
The existence of $\underline{f}$ with the correct properties follows
from~\cite[Lemma 5.4]{Esc13}, except that monotonicity is not proven. Let
us thus construct explicitly $\underline{f}$.

First, note that inflationary maps $f : \bN \to \bN$ are in one-to-one
correspondence with maps $f' : \bN \to \bN$ by taking $f(n) = n + f'(n)$.
We will thus now work explicitly with $f'$ from now on.
We define a coalgebra map
\[
\begin{array}{llcl}
  c : & \Baire \times \Ninfty &\longto&  1 + \Baire \times \Ninfty \\
      & (f', \underline{0}) &\longmapsto& \left\{ \begin{array}{ll}
                                                  \inl(0) & \text{if $f'(0) = 0$} \\
                                                  \inr(f' - 1, \underline{0}) & \text{otherwise ($-1$ is the pointwise operation)}
                                                \end{array} \right. \\
      & (f', \Succinfty(x)) &\longmapsto& \inr(f' \circ \Succ, x)
\end{array}
\]
Thus, by~\Cref{lem:NinftyTerminalCoAlg} and currying, we have a map
$H : \Baire \to \Ninfty^{\Ninfty}$ such that $H(f')(\underline{0}) = \underline{f'(0)}$
(this can be proven by induction on $f'(0)$) and
$H(f')(\Succinfty(x)) = \Succinfty(H(f' \circ \Succ)(x))$. We take $\underline{f} = H(f')$.

For every $x \in \Ninfty$ such that $x \ge \underline{n}$, by induction over $n$, we get
\[ H(f')(x) = \underline{n} + H(f' \circ \Succ^n)(x - n)\]
For $x = \underline{n}$, this yields $\underline{f}(\underline{n}) = H(f')(\underline{n}) = \underline{n + f'(n)} = \underline{f(n)}$,
and for $x = \infty$, this gives $\underline{f}(\infty) = H(f')(\infty) \ge \underline{n}$
for every $n \in \bN$ and thus $\underline{f}(\infty) = \infty$.

By induction over $n$, one can show that $H(f')(x) \le \underline{n}$ implies
$x \le \underline{n}$ for any $f'$: if $x = \underline{0}$ this is trivial,
$n = 0$ entails $x = \underline{0}$ and otherwise, we can use the induction
hypothesis with $f' \circ \Succ$ in lieu of $f'$.

As a consequence, one can deduce monotonicity: to establish that, by~\Cref{lem:NinftyOrdYon}
it suffices to prove $\underline{f}(y) \ge \underline{n}$ from $x \le y$ and $\underline{f}(x) \ge \underline{n}$
for arbitrary $x,y \in \Ninfty$ and $n \in \bN$. By~\Cref{lem:minwoMP}, assume
without loss of generality that $\underline{f}(y) < \underline{n}$.
But then we have $y \le \underline{n}$ and a fortiori $x \le \underline{n}$.
Thus there are some $n_x \le n_y \in \bN$ such that
$x = \underline{n_x}$, $y = \underline{n_y}$, and we can derive the following absurd chain of inequalities:
\[ \underline{n} > \underline{f}(y) = \underline{f(n_y)} \ge \underline{f(n_x)} = \underline{f}(x) \ge \underline{n} \]

Now assume that $f$ is strictly increasing and let us prove $\underline{f}$ is
an embedding. To this end, conversely assume that $\underline{f}(x) \le \underline{f}(y)$,
$x \ge \underline{n}$ and $y < \underline{n}$ towards a contradiction (employing~\Cref{lem:minwoMP} again).
We have again $n_y \in \bN$ such that $y = \underline{n_y}$ and
$\underline{f}(y) = \underline{f(n_y)} \ge \underline{f}(x)$. Hence, we have
$x \le \underline{f(n_y)} \in \underline{\bN}$, so there is some $n_x \in \bN$
with $x = \underline{n_x}$ and $\underline{f}(x) = \underline{f(n_x)}$.
Because $n \mapsto \underline{n}$ is an embedding, we have
$f(n_x) \le f(n_y)$. Since $f$ is strictly increasing, we have $n_x \le n_y$.
But then we have the absurd chain of inequalities $n > y \ge x \ge n$.
\end{proof}

\subsection{The Myhill isomorphism theorem}

We refer to the introduction for the key definitions; here we give a couple of
intuitions on how the usual Myhill isomorphism theorem is proven and a couple
of useful facts.

A first helpful intuition is that the sets $A$ and $B$ do not matter as data
when computing the bijection in the conclusion of the theorem. The relevant
property is enforced by making sure that the bijection only relates pairs which
belong to the transitive closure of the following bipartite graph, where
$f$ and $g$ are injections witnessing $A \preceq_1^X B$ and $B \preceq_1^X A$
respectively.
\[ G_{f,g} = \{ (\inl(n), \inr(f(n))) \mid n \in \mathbb{N} \} \cup \{ (\inl(g(m)), \inr(m)) \mid m \in \bN\} ~ \subseteq ~ (\bN + \bN)^2\]
Since $f$ and $g$ are reductions, whenever we have an edge
$(u,v)$ in $G_{f,g}$, then $u$
belongs to the subgraph induced by $A$ and $B$ (i.e., $\inl(A) \cup \inl(B)$)
if and only if $v$ does.

\begin{definition}
We say that $X$ has the \emph{strong Myhill property} if for every pair of injections
$f, g : X \to X$, there is a bijection $h : X \to X$ such that
$h \subseteq \bigcup_{m \in \bZ} f \circ (g \circ f)^{m}$.
\end{definition}

It is clear that the strong Myhill property implies the Myhill property for any
set $X$. We leave the converse as an open question.

So now, let us explain in more details how we can prove the Myhill isomorphism
theorem.

\begin{theorem}[{essentially~\cite[Theorem 18]{Myhill55}}]
  \label{thm:myhilliso55}
$\mathbb{N}$ has the strong Myhill property.
\end{theorem}
\begin{proof}[Proof idea]
Call
a finite subset\footnote{I.e., isomorphic to
an ordinal. More concretely, this can be represented by a list of pairs of numbers.}
$I$ of $\bN^2$ such that
\[
(n, m) \in I ~\text{and}~ (n', m') \in I \qquad \text{implies} \qquad n = n' \Longleftrightarrow m = m'\]
  a \emph{partial isomorphism} over $\bN$.
We can build an increasing sequence $I_n$
of such partial isomorphisms such that, for every $n \in \bN$, we have

\begin{enumerate}
  \item \label{enumitem:partialIsoInGraph} $I_n \subseteq \bigcup\limits_{m \in \bZ} f \circ (g \circ f)^{m}$
\item for every $k < n$, $k$ belongs to both the domain and codomain of $I_n$
\end{enumerate}

\begin{figure}[h]
\[\begin{tikzcd}[cramped,column sep=small]
	& \vdots && {} & \vdots &&&&& \vdots & {} & {} & \vdots \\
	& 4 &&& 4 &&&&& 4 &&& 4 \\
	& 3 &&& 3 &&&&& 3 &&& 3 \\
	& 2 &&& 2 &&&&& 2 &&& 2 \\
	& 1 &&& 1 &&&&& 1 &&& 1 \\
	& 0 &&& 0 &&&&& 0 &&& 0 \\
	\textcolor{rgb,255:red,184;green,6;blue,0}{f} & {\mathbb{N}} &&& {\mathbb{N}} &
    \textcolor{rgb,255:red,51;green,65;blue,255}{g}
    &&& h & {\mathbb{N}} &&& {\mathbb{N}}
	\arrow[color={rgb,255:red,51;green,65;blue,255}, from=2-5, to=3-2]
	\arrow["{\qquad\footnotesize43f(3)g(f(3))f(g(f(3)))}"{description, pos=0.3}, dashed, no head, from=2-13, to=1-11]
	\arrow[color={rgb,255:red,184;green,6;blue,0}, dashed, no head, from=3-2, to=1-4]
	\arrow[color={rgb,255:red,51;green,65;blue,255}, from=3-5, to=4-2]
	\arrow["{3f(3)}"{description, pos=0.2}, dashed, no head, from=3-10, to=1-12]
	\arrow["3224"{description, pos=0.2}, from=3-13, to=2-10]
	\arrow[color={rgb,255:red,184;green,6;blue,0}, from=4-2, to=4-5]
	\arrow[color={rgb,255:red,51;green,65;blue,255}, from=4-5, to=2-2]
	\arrow["22"{description, pos=0.3}, from=4-10, to=4-13]
	\arrow[color={rgb,255:red,184;green,6;blue,0}, from=5-2, to=6-5]
	\arrow[color={rgb,255:red,51;green,65;blue,255}, from=5-5, to=5-2]
	\arrow[color={rgb,255:red,184;green,6;blue,0}, from=6-2, to=5-5]
	\arrow[color={rgb,255:red,51;green,65;blue,255}, from=6-5, to=6-2]
	\arrow["01"{description, pos=0.2}, from=6-10, to=5-13]
	\arrow["0011"{description, pos=0.2}, from=6-13, to=5-10]
	\arrow["{:}"{description}, color={rgb,255:red,184;green,6;blue,0}, draw=none, from=7-1, to=7-2]
	\arrow[shift right, color={rgb,255:red,184;green,6;blue,0}, shorten <=8pt, shorten >=8pt, from=7-2, to=7-5]
	\arrow[shift right, color={rgb,255:red,51;green,65;blue,255}, shorten <=8pt, shorten >=8pt, from=7-5, to=7-2]
	\arrow["{:}"{description}, color={rgb,255:red,51;green,65;blue,255}, draw=none, from=7-6, to=7-5]
	\arrow["{:}"{description}, draw=none, from=7-9, to=7-10]
	\arrow[shorten <=8pt, shorten >=8pt, from=7-10, to=7-13]
\end{tikzcd}\]
\caption{Construction of the bijection from the proof that $\bN$ has the (strong)
  Myhill property. On the left is pictured a bipartite graph corresponding
  to injections $f, g : \mathbb{N} \to \bN$ ($f$ goes left-to-right while $g$
  goes right-to-left), and on the right is the
  graph of the isomorphism $h = \bigcup_{n \in \mathbb{N}} I_n$. The edges are
  directed to indicate whether an edge was added during a forward or backward
step, and labelled by the paths they correspond to in the graph on the left.}
\label{fig:myhilliso55}
\end{figure}

We can of course set $I_0 = \emptyset$. To go from $I_n$ to $I_{n+1}$, we
potentially need to add $n$ to the domain and/or the codomain if it is not
there already (we can check if it is in either). If it is not in the domain,
we can decide for every $k \le n$ whether $(f \circ (g \circ f)^k)(n)$ is in the
codomain of $I_n$ or not.
Since $I_n$ has size $n$, we know there is a minimal such $k$ which isn't.
Set $I'_n := I_n \cup \{(n, (f \circ (g \circ f)^k)(n))\}$ and note it is still a partial isomorphism.
We can similarly add $n$ to the codomain of $I_n$ or $I'_n$ if required and
complete the definition of $I_{n+1}$. We picture the construction in~\Cref{fig:myhilliso55}.

Then it can be shown that $h = \bigcup_{n \in \mathbb{N}} I_n$ is
the desired isomorphism.
\end{proof}

We will reuse a variation of this construction in~\Cref{sec:positive}.
For some of our negative results, we will also use the fact
that subfinite
sets have the Myhill property.

\begin{lemma}
  \label{lem:subfinite}
For any finite set $n \in \bN$ and $A \subseteq n$, $A$ has the strong Myhill
property.
\end{lemma}
\begin{proof}
This is trivial for $n = 0$. Otherwise, assuming we have injections
$f, g : A \to A$, the desired bijection is $f$ with inverse
$g \circ (f \circ g)^{(n-1)!}$ ($f \circ g$ is a partial bijection over $n$ and
thus orbits of $x \in A$ have size less than $n$).
\end{proof}

\section{Negative results about the (strong) Myhill property}
\label{sec:negative}

\subsection{On the conatural numbers}

Our first negative result is that it is impossible to prove constructively
that $\Ninfty$ has the strong Myhill property. To do so, we will use that it is
incompatible with continuity and a slight extension of countable choice.

\begin{definition}
For any set $X$, we say that $X$-choice holds if and only if every surjection
$r : Y \to X$ has a section $s : X \to Y$ (i.e. $r \circ s$ is the identity).
\end{definition}

\begin{remark}
For any fixed object $X$ in a topos, $X$-choice holds precisely
when $X$ is an internally projective object. In relative realizability,
the internally projective objects are those isomorphic to partitioned assemblies.
In the case of Kleene-Vesley realizability, this includes include all (classical)
subsets of $\Cantor$. In particular, $\Ninfty$-choice holds in Kleene-Vesley
realizability.
\end{remark}

With $X$-choice, we then know that the Myhill property for $X$ gives us
bijections $h : X \cong X$ generated by maps $\iota : X \to \bZ$
telling us how far in the graph and in which direction to travel to build $h$.

\begin{lemma}
  \label{lem:choiceToZ}
Assume $X$-choice holds. If $X$ has the strong Myhill property, then for every
injections $f, g : X \to X$, there is $\iota : X \to \bZ$ such that the
following map is well-defined and a bijection:
\[
\begin{array}{llcl}
h :& X &\longto& X \\
& x &\longmapsto& \left(f \circ (g \circ f)^{\iota(x)}\right)(x)
\end{array}
\]
\end{lemma}
\begin{proof}
The bijection $h$ one gets from the strong Myhill property satisfies
\[\forall x \in X. \exists m \in \bZ.~~ h(x) = \left(f \circ (g \circ f)^{m}\right)(x)\]
Rewriting this $\forall\exists$ to $\exists\forall$ using $X$-choice yields the
function $\iota$ with the expected property.
\end{proof}

To exploit this, we will find two injections that will force $\iota$
to infinitely many times go forward and backwards in the graph, which is
impossible to do continuously if the domain of $\iota$ is $\Ninfty$
(see~\Cref{fig:NinftyMPLPO} for an intuition of how we force this).

\begin{lemma}
  \label{lem:oscillation}
$\LPO$ holds if and only if $\MP$ holds and there exists a map $f : \Ninfty \to 2$ such that
$f(\underline{n})$ is different from $f(\infty)$ infinitely many times (i.e. $\forall n \in \bN. \; \exists
k \in \bN. f(\underline{n + k}) \neq f(\infty)$).
\end{lemma}
\begin{proof}
The direct implication is easy since $\LPO$ yields a bijection
$h : \Ninfty \cong \bN$: take $f(x) = h(x) \mod 2$.
To prove the converse, by~\Cref{lem:NinftyLPO}, it suffices to show we can decide
whether $x = \infty$ or $x \in \underline{\bN}$ for an arbitrary $x \in \Ninfty$.
By~\Cref{thm:Ninftysel}, the following is a decidable predicate:
\[\exists y \in \Ninfty. \; f(x + y) \neq f(\infty)\]
If it holds, then $x \neq \infty$ as otherwise we'd have $x + y = \infty$; we can
then conclude that $x \in \underline{\bN}$ by Markov's principle (\Cref{lem:NinftyMP}).
If it does not hold, we have $x \ge \underline{n}$ for every $n \in \bN$.
If it were not the case, then we would have $x = \underline{n_x}$ for some $n_x \in \bN$ (this can be shown by induction on $n$).
Since $f$ oscillates infinitely many times, there would then be some $k$ such
that $f(x + \underline{k}) = f(\underline{n_x + k}) \neq f(\infty)$, a contradiction.
\end{proof}

\begin{theorem}
  \label{thm:oscNinftyNotMP}
If $\Ninfty$-choice and $\MP$ hold, $\Ninfty$ has the strong Myhill property if and only if $\LPO$ holds.
\end{theorem}
\begin{proof}
Let us focus on the non-trivial left-to-right implication.
Write $\tuple{-, -} : \bN^2 \to \bN$ for the standard bijection
$\displaystyle (n,k) \mapsto \frac{(n + k)(n+ k + 1) + 2n}{2}$.
Let us consider the following two maps, which we picture in~\Cref{fig:NinftyMPLPO}.
\[
  \begin{array}{llcl !\qquad llcl}
f : & \bN &\longto& \bN
    &
g : & \bN &\longto& \bN
\\
    & \tuple{2n,m} &\longmapsto& \tuple{2n, m + 1}
    &
    & \tuple{2n,m} &\longmapsto& \tuple{2n, m}
\\
    & \tuple{2n+1,m} &\longmapsto& \tuple{2n+1, m}
    &
    & \tuple{2n+1,m} &\longmapsto& \tuple{2n+1, m + 1}
\end{array}
\]

Clearly, $f$ and $g$ are strictly increasing injections, so we can lift them
to corresponding injections $\underline{f}, \underline{g} : \Ninfty \to \Ninfty$
by~\Cref{lem:liftInflationary}. Apply~\Cref{lem:choiceToZ} to $\underline{f}$
and $\underline{g}$ to get the corresponding map $\iota : \Ninfty \to \bZ$
such that $h(x) = \left(\underline{f} \circ (\underline{g} \circ \underline{f})^{\iota(x)}\right)(x)$ is a well-defined bijection.

\begin{figure}[h]
\[\begin{tikzcd}[cramped,column sep=small,row sep=small]
	&&&&&&&&&&& {} \\
	&&&&&&&&& \vdots & {} & {} & \vdots \\
	&&&&&&&&& {\langle0,2\rangle} &&& {\langle0,2\rangle} \\
	& {} && {} &&&&&& {\langle1,1\rangle} &&& {\langle1,1\rangle} \\
	2 & 2 && 2 & 2 &&&&& {\langle2,0\rangle} &&& {\langle2,0\rangle} \\
	1 & 1 && 1 & 1 & {} &&& {} & {\langle0,1\rangle} &&& {\langle0,1\rangle} \\
	0 & 0 && 0 & 0 &&&&& {\langle1,0\rangle} &&& {\langle1,0\rangle} \\
	{} & {} && {} & {} &&&&& {\langle0,0\rangle} &&& {\langle0,0\rangle} \\
	&&&&&&&& {f} & {\mathbb{N}} &&& {\mathbb{N}} & {g}
	\arrow[color={rgb,255:red,153;green,92;blue,214}, dashed, no head, from=3-10, to=2-11]
	\arrow[color={rgb,255:red,153;green,92;blue,214}, from=3-13, to=3-10]
	\arrow[color={rgb,255:red,168;green,31;blue,0}, from=4-10, to=4-13]
	\arrow[color={rgb,255:red,168;green,31;blue,0}, dashed, no head, from=4-13, to=2-12]
	\arrow[dashed, no head, from=5-1, to=4-2]
	\arrow[from=5-2, to=5-1]
	\arrow[from=5-4, to=5-5]
	\arrow[dashed, no head, from=5-5, to=4-4]
	\arrow[color={rgb,255:red,23;green,94;blue,23}, dashed, no head, from=5-10, to=1-12]
	\arrow[color={rgb,255:red,23;green,94;blue,23}, from=5-13, to=5-10]
	\arrow[from=6-1, to=5-2]
	\arrow[from=6-2, to=6-1]
	\arrow[from=6-4, to=6-5]
	\arrow[from=6-5, to=5-4]
	\arrow["{\text{$\infty$ many times}}", squiggly, maps to, from=6-6, to=6-9]
	\arrow[color={rgb,255:red,153;green,92;blue,214}, from=6-10, to=3-13]
	\arrow[color={rgb,255:red,153;green,92;blue,214}, from=6-13, to=6-10]
	\arrow[from=7-1, to=6-2]
	\arrow[from=7-2, to=7-1]
	\arrow[from=7-4, to=7-5]
	\arrow[from=7-5, to=6-4]
	\arrow[color={rgb,255:red,168;green,31;blue,0}, from=7-10, to=7-13]
	\arrow[color={rgb,255:red,168;green,31;blue,0}, from=7-13, to=4-10]
	\arrow["{\text{even ladder}}"{description}, draw=none, from=8-1, to=8-2]
	\arrow["{\text{odd ladder}}"{description}, from=8-4, to=8-5]
	\arrow[color={rgb,255:red,153;green,92;blue,214}, from=8-10, to=6-13]
	\arrow[color={rgb,255:red,153;green,92;blue,214}, from=8-13, to=8-10]
	\arrow["{{:}}"{description}, draw=none, from=9-9, to=9-10]
	\arrow[shift right, shorten <=8pt, shorten >=8pt, from=9-10, to=9-13]
	\arrow[shift right, shorten <=8pt, shorten >=8pt, from=9-13, to=9-10]
	\arrow["{{:}}"{description}, draw=none, from=9-14, to=9-13]
\end{tikzcd}\]
\caption{Representation of the bipartite graph generated by $f$ and $g$,
which is intuitively obtained by interleaving infinitely many ``ladders'' as pictured
on the left. Up to adding $\infty$s, this graph also represents $\underline{f}$
and $\underline{g}$. The map $\iota$ induced by the Myhill property cannot
always say ``go forward'' on an even ladder (and vice-versa for ``go backwards'' and
the odd ladders), which can be used to deduce $\LPO$.}
\label{fig:NinftyMPLPO}
\end{figure}

First note that for any $n \in \bN$ and $x \in \Ninfty$, we can prove by
induction over $|\iota(x)|$ that we have
\[\left(\exists m \in \bN. \; h(x) = \underline{\tuple{n,m}}\right)
\quad \Longleftrightarrow \quad \left(\exists m' \in \bN. \; x = \underline{\tuple{n,m'}}\right)
\qquad \qquad (\dagger)\]

Define $\mathrm{sgn} : \bZ \to 2$ as the map with $\mathrm{sgn}(d) = 1$ if and
only if $d \ge 0$. By~\Cref{lem:oscillation}, it suffices to show that
$\mathrm{sgn} \circ \iota$ takes the values $0$ and $1$ infinitely often
over $\underline{\bN}$ to conclude that $\LPO$ holds. To this end, fix some $a \in \bN$, $i \in 2$ and let us
look for some values $b \ge a$ with $\mathrm{sgn}(\iota(b)) = i$.
First note that for any such $a$, there is $n_0 \in \bN$ such that
\[\forall n \ge n_0. \; \forall m \in \bN. \; a \le \tuple{n, m}\]
Pick $n$ such that $2n + i \ge n_0$. We make a case distinction on $i$:
\begin{itemize}
\item If $i = 0$, by~$(\dagger)$, $h^{-1}(\underline{\tuple{2n,0}}) = \underline{\tuple{2n, m}}$
  for some $m \in \bN$.
  Set $d = \iota(\underline{\tuple{2n,m}})$ and assume towards a contradiction
  that $d \ge 0$. By 
  definition $h(\underline{\tuple{2n,m}}) = (\underline{f} \circ \underline{g})^d(\underline{f}(\underline{\tuple{2n, m}}))$. Using the fact that
  $\underline{f}$ and $\underline{g}$ are liftings, we thus have
  $ \underline{(f \circ g)^d(\tuple{2n,m+1})} =
    h(\underline{\tuple{2m,m}}) =
    \underline{\tuple{2n,0}}$.
  But we have that $f \circ g$ is increasing, so we cannot have $d \ge 0$.
  So $d < 0$ and $\mathrm{sgn}(\iota(h^{-1}(\underline{\tuple{2n,0}}))) = 0$.
\item If $i = 1$, set $d = \iota(\underline{\tuple{2n+1, 0}})$, so
  that $h(\underline{\tuple{2n+1,0}}) = \left(\underline{f} \circ (\underline{g} \circ \underline{f})^d\right)(\tuple{2n+1,0})$.
  Towards a contradiction, assume $d < 0$; this means in particular that
  $\tuple{2n+1,0}$ is in the image of $g$
  and we can set $z = \underline{g}^{-1}(\tuple{2n+1,0})$. We cannot
  have $z = \infty$, for it would contradict injectivity of $\underline(g)$.
  So by $\MP$, we have some $n', m' \in \bN$ such that $z = \underline{\tuple{n',m'}}$.
  But then we have $\underline{g}(\tuple{n',m'}) = \underline{g(\tuple{n',m'})} = \underline{\tuple{2n+1,0}}$, which is a contradiction as
  $\tuple{2n+1,0}$ is not in the image of $g$. Hence
  $d \ge 0$ and $\mathrm{sgn}(\iota(\underline{\tuple{2n+1,0}})) = 1$.
\end{itemize}
\end{proof}

We can use that to prove that $\Ninfty$ does not have the Myhill property in
the setting of Kleene-Vesley realizability. The basic idea is that a realizer
for some $X$ having the Myhill property can uniformly be turned into one
for $X$ having the strong Myhill property if we assume further that $X$ is
tame enough. 
The notion of tameness we need involves the class of those partioned assemblies
that are also modest sets; we recall the definition of the notion
(following e.g.~\cite{VanOosten}), and will
otherwise assume familiarity with Kleene-Vesley realizability for the proof.

\begin{definition}
An assembly $Z = (|Z|, \Vdash_Z)$ consists of a carrier set $|Z|$ and
a surjective relation $\Vdash_Z \subseteq \mathcal{K}_2 \times |Z|$ (i.e.,
such that for every $z \in |Z|$ there exists at least one realizer $e \in \mathcal{K}_2$
  such that $e \Vdash_Z z$).

$Z$ is a modest set if we additionally have that each realizers realizes at
most one element.

$Z$ is a partitioned modest set if it is a modest set and each element $z \in |Z|$
is realized by at most one realizer (so that $\Vdash_Z$ is a partial
isomorphism which is surjective).
\end{definition}

\begin{lemma}
\label{lem:transfermodestcountable}
Working classically (and meta-theoretically), let $X$ be a partitioned modest set in the Kleene-Vesley topos with 
countable carrier $|X|$.
Then if ``$X$ has the Myhill property'' holds in Kleene-Vesley realizability, so does ``$X$ has the strong Myhill property''.
\end{lemma}
\begin{proof}
For this proof, write $e \Vdash \varphi$ for ``$e$ realizes $\varphi$'' and $\tuple{ -, -, \ldots}$ for the standard tupling in a PCA. We will also use the
$\lambda$ notation when convenient to work in $\mathcal{K}^{\mathrm{rec}}_2$.
Assume we have a realizer $e_{\mathrm{My}} \in \mathcal{K}_2^{\mathrm{rec}}$ for ``$X$ has the Myhill property''.
Shuffling around the arguments of $e_{\mathrm{My}}$ and suppressing those that have trivial computational content due to $X$ being a modest set,
we have that whenever we are given
\begin{itemize}
\item 
injections $f, g : |X| \to |X|$,
\item realizers
  $e_f \Vdash_{X \to X} f$ and $e_g \Vdash_{X \to X} g$,
\item $A, B : |X| \to \powerset(\mathcal{K}_2)$
\item and
realizers
$e_{f,\preceq} \Vdash \forall x \in X. \; x \in A \Leftrightarrow f(x) \in B$ and
$e_{g,\preceq} \Vdash \forall x \in X. \; x \in B \Leftrightarrow g(x) \in A$
\end{itemize}
then we have some bijection $h : |X| \to |X|$ and codes
\[e_h \Vdash_{X \to X} h \qquad e_{h^{-1}} \Vdash_{X \to X} h^{-1} \quad \text{and} \quad e_{h, \preceq} \Vdash \forall x \in X. \; x \in A \Rightarrow h(x) \in B\] such that $e_{\mathrm{My}} \; e_f \; e_g \; e_{f, \preceq} \; e_{g, \preceq} \simeq \tuple{e_h, e_{h^{-1}}, e_{h,\preceq}}$.

We will use this to build a realizer $e_{\mathrm{sMy}} \in \mathcal{K}_2^{\mathrm{rec}}$ that $X$ has the strong Myhill property, that is, we should have for every $e_f \Vdash_{X \to X} f$ and $e_g \Vdash_{X \to X} g$ for injections $f, g : |X| \to |X|$ that there is a bijection $h : |X| \to |X|$
and codes $e_h \Vdash_{X \to X} h$, $e_{h^{-1}} \Vdash_{X \to X} h^{-1}$
and $e_{h, \mathrm{sMy}}$ such that, for every $e_x \Vdash_X x$, there is $m \in \bZ$ such that  $e_{h, \mathrm{sMy}} \; e_x \Vdash_{\bZ} m$ and $h(x) = f
\circ (g \circ f)^m(x)$.
To then argue that our proposed realizer is sound,
we will exploit correctness of $e_{\mathrm{My}}$ for different predicates $A_{e_f, e_g}, B_{e_f, e_g} : |X| \to \powerset(\mathcal{K}_2)$, where $e_f, e_g \in \mathcal{K}_2$ are codes for injections $f, g : |X| \to |X|$.

Concretely $e_{\mathrm{sMy}}$ is defined by
\[ e_{\mathrm{sMy}} = \lambda f. \lambda g. p \; \left(e_{\mathrm{My}} \; f \; g \; e_0 \; e_1\right) \]
where $p$, $e_0$ and $e_1$ are characterized as follows:
\begin{itemize}
\item $e_0$ is a code for the function $X \longto \left( \bZ \times \Baire \to \bZ \times \Baire\right)^2$
which ignores its argument and produces the following two functions:
\[(m, p) \longmapsto (m+1,p) \qquad (m, p) \longmapsto (m-1,p)\]
\item similarly, $e_0$ is a code for a function $X \longto \left( \bZ \times \Baire \to \bZ \times \Baire\right)^2$
that ignores its argument and produces a code for two functions, but this time, the functions in question are both the identity.
\item $p$ codes a map $\id_{\mathcal{K}_2} \times \id_{\mathcal{K}_2} \times p' : (\mathcal{K}_2)^3 \to (\mathcal{K}_2)^3$ that acts componentwise. The
first two components are the identity, but the last one is non-trivial. $p'$ takes as input a code for map $\iota' : X \to \left(\bZ \times \Baire \partto
\bZ \times \Baire\right)^2$ and outputs a code for a map $X \partto \bZ$ as follows: the $X$ argument $x$ of $p'$ is fed to $\iota'$
unchanged and we take a code for the corresponding first component $a_x : \bZ \times \Baire \partto \bZ \times \Baire$. Define the map $a'_x : \Baire \partto \bZ$ by $a'_x(s) = \pi_1(a_x(0,s))$. What $p'$ applied to a
code for $x$ does is looking for any computable\footnote{It would be sufficient to restrict to sequences that
are ultimately $0$.} $s \in \Baire$ such that $a'_x$ is well-defined and outputs
$a'_x(s)$.
\end{itemize}

Now let's see why this realizer is correct. Assume we have $f, g : |X| \to |X|$ injective and $e_f \Vdash_{X \to X} f$ and $e_g \Vdash_{X \to X} g$. Using countability of $|X|$,
fix a map $C : |X| \to \powerset(\mathcal{K}_2)$ such that $(C(x))_{x \in |X|}$
is an antichain in the Turing degrees relative to the infinite tupling of
$e_f$, $e_g$ and the countably many\footnote{This set is countable because
$|X|$ is countable and $X$ is partitioned modest.} $e_x$s such that $e_x \Vdash_X x$ for some $x \in |X|$
 (that is, if we take this tupling as oracle, we can't compute Turing reductions
between $C(x)$ and $C(y)$ whenever $x \neq y$). $C$ corresponds, in the semantics of Kleene-Vesley
realizability, to a subset of $X$. We then define $A : |X| \to \powerset(\mathcal{K}_2)$ such that $A(x)$
consists of codes $\tuple{e, e'}$ such that $e \Vdash_\bZ m$ for some $m \in \bZ$ and $e' \in C_{(f \circ g)^m(x)}$.
$B : |X| \to \powerset(\mathcal{K}_2)$ is then simply defined by $B(x) = A(g(x))$.

The basic intuition is that, for providing evidence that $x$ belongs to $A$ (or $B$), one must provide some evidence that we can compute the degree $C(y)$
  associated with some $y$ in the same connected component in the bipartite graph generated by $f$ and $g$.
One also must provide some $m \in \bZ$ that tells us how to travel from $y$ to $x$ in the graph. With this in mind, it is not difficult to check that we
  have some
\[e_0 \Vdash \forall x \in X. \; x \in A \Leftrightarrow f(x) \in B \qquad \text{and}\qquad
e_1 \Vdash \forall x \in X. \; x \in A \Leftrightarrow g(x) \in B
\]
We thus know that $e_{\mathrm{My}} \; e_f \; e_g \; e_0 \; e_1 \simeq \tuple{e_h, e_{h^{-1}}, e_{h, \preceq}}$
with $e_h$ and $e_{h^{-1}}$ describing the bijection $h$ we desire and $e_{h, \preceq} \Vdash \forall x \in X. \; h(x) \in A \Leftrightarrow x \in B$. Now,
we simply need to check that
$p' \; e_{h, \preceq} \; e_x$ computes a code for some $m \in \bZ$ such that $h(x) = f \circ (g \circ f)^m(x)$
whenever $e_x \Vdash_X x$.
Here the procedure encoded in $p'$ first produces some $e' \in \mathcal{K}_2$ which is computable
from $\tuple{e_f, e_g, e_x}$ so that
\[e' \Vdash x \in A \Rightarrow h(x) \in B\]
More concretely, this means that $e'$ represents a map
$\alpha : \bZ \times C(x) \to \bZ \times C(x)$ such that 
$(\pi_1 \circ \alpha)(n,s) = m$ implies that $h(x) = f \circ (g \circ f)^{m-n}(x)$. Ultimately,
$p'$ searches for some approximation of $s$ that makes the code output some $m$ for $n = 0$.
Since convergence must happen after finding a finite prefix of $s$, $p'$ will find such a $k$.
But since $C(x) \subseteq \Baire$ is closed under Turing-equivalence, it is dense in $\Baire$, so the prefix used must be consistent
with some $s \in C(x)$ and thus the obtained $k$ satisfies $h(x) = f \circ (g \circ f)^k(x)$.
\end{proof}

\begin{corollary}
\label{cor:KVnotNinftyMP}
Kleene-Vesley realizability does not validate that $\Ninfty$ has the Myhill property.
\end{corollary}
\begin{proof}
$\Ninfty$ is a partitioned modest set in Kleene-Vesley realizability with a countable carrier.
So if it had the Myhill property, it would have the strong Myhill property by~\Cref{lem:transfermodestcountable}.
But this is not the case by~\Cref{thm:oscNinftyNotMP}.
\end{proof}

\subsection{Failure of stability under exponentiation, product and coproducts}
\label{subsec:negEM}

Here we show
that if the class of sets with the Myhill property
is closed under either products, exponentiations or coproducts, $\LPO$ implies
full excluded middle, an implication which is not provable constructively~\cite{HL16}.
This relies on sets parameterized by truth values rather than concrete counter-examples
much like in the proofs that Cantor-Bernstein implies excluded middle
in~\cite{CBEM}. 

\begin{lemma}
  \label{lem:expEM}
For any $p \in \Omega$, if $(\{0,1\} \uplus \{2 \mid p\})^\bN$ has the Myhill property,
then $p \vee \neg p$ holds.
\end{lemma}
\begin{proof}
Take $A = \{x + x \mid x \in \Ninfty\}$ and $B = A \cup \{ 20^\omega \mid p\}$.
Reductions between $A$ and $B$ are witnessed by the injections $f$ and $g$
defined by:
\[f(s) = 00s \qquad \qquad g(2s) = 1s \quad \text{and} \quad g(s) = 00s ~~\text{otherwise}\]
These two functions are easily checked to be injections and we do have
$f^{-1}(B) = A$ and $g^{-1}(A) = B$. Now assume that, thanks to the Myhill
property, we have a bijection $h : 
(\{0,1\} \uplus \{2 \mid p\})^\bN
\to
(\{0,1\} \uplus \{2 \mid p\})^\bN$
such that $h(A) = B$.
By \Cref{thm:Ninftysel}, we can decide whether there exists
$x \in \Ninfty$ such that $h(x)_0 = 2$ or not.
If there is, then we know that $p$ must hold. Otherwise, assuming $p$,
we know that $2_\infty \in B$. But then since $B = h(A)$ and $A \subseteq \Ninfty$,
we would have a contradiction. Hence $\neg p$ holds in that case.
\end{proof}
\begin{corollary}
  \label{cor:coproductMyhill}
``$A^B$ has the Myhill property whenever $A$ and $B$ both have the strong Myhill property''
entails excluded middle.
\end{corollary}
\begin{proof}
Combine the previous lemma with~\Cref{lem:subfinite} and Myhill's isomorphism
theorem.
\end{proof}

\begin{lemma}
For any $p \in \Omega \subseteq 1$, if $p + \bN$ has the Myhill property
and $\LPO$ holds, then $p \vee \neg p$ holds.
\end{lemma}
\begin{proof}
Take $A = \{\inr(2n) \mid n \in \bN\}$, $B = \{\inl(0) \mid p\} \cup A$ and
$f, g : p + \bN \to p + \bN$ defined by
\[
\begin{array}{lcl!\qquad lclr}
  f(\inr(k)) &=& \inr(k + 2) & g(\inr(k)) &=& \inr(k+2) \\
  f(\inl(0)) &=& \inr(1) & g(\inl(0)) &=& \inl(0) & %
\end{array}
\]
Using the Myhill property, we have a bijection $h : p + \bN \to p + \bN$
with $h(A) = B$. Using $\LPO$, we can decide whether the sequence
$n \mapsto h(\inr(2n))$ ever takes the value $\inl(0)$ (which is equivalent
to deciding whether $\inl(0) \in h(A)$). If we do have
$h(\inr(2n)) = \inl(0)$, we do know that $p$ holds. Otherwise, assuming $p$,
we know that $\inl(0) \in p + \bN$, and we thus should have $h^{-1}(\inl(0)) \in A$,
which contradicts $\inl(0) \not\in h(A)$; hence $\neg p$ holds.
\end{proof}

\begin{lemma}
For any $p \in \Omega$, if $(p \uplus \{1\}) \times \bN$ has the Myhill property and $\LPO$ holds,
then $p \vee \neg p$ holds.
\end{lemma}
\begin{proof}
Take $A = \{ (1, 2x) \mid x \in \bN\}$ and $B = A \cup \{(0, 0) \mid p\}$.
Then we have injections 
$f$ and $g$ that witness $A \preceq_1^{(p \uplus \{1\}) \times \bN} B$ and $B \preceq_1^{(p \uplus \{1\}) \times \bN} A$
respectively:
\[
\begin{array}{llcl !\qquad lcl}
& f(1, z) &=& (1,2 + z) &
g(1,z) &=& (1, 2 + z) \\
 \text{and, if $p$ holds, } &
 f(0, x) &=& (0, x + 1) &
 g(0, {0}) &=& (1, {0}) \\
& & & & g(0, x + {1}) &=& (0, x) \\
\end{array}
\]
Now, if $p + \bN$ has the Myhill property, then we have a bijection $h :
(p \uplus \{1\}) \times \bN \to (p \uplus \{1\}) \times \bN$
with $h(A) = B$. 
By $\LPO$, we can decide whether there exists
$x \in \bN$ such that $h(1,x) = (0, {0})$ or not.
If there is, then we know that $p$ must hold.
Otherwise, assuming $p$, we have $(0, {0}) \in B$.
But since $B = h(A)$ and $A \subseteq \{1\} \times \Ninfty$, we then would
have a contradiction. Hence $\neg p$ holds in that case.
\end{proof}

\begin{corollary}
  \label{cor:coproductMyhill}
``$A+B$ or $A \times B$ has the Myhill property whenever $A$ and $B$ both have the strong Myhill property''
together with $\LPO$ entails excluded middle.
\end{corollary}

\section{Weakening the Myhill property and one positive result}
\label{sec:positive}

\subsection{The modal Myhill property}

In the rest of this paper, we discuss
a weakening of the Myhill property, as well as a corresponding weakening of
the strong Myhill property. The basic idea is that, while we are still going to
ask explicitly for a bijection, we will relax our requirement on its correctness
so that we may no longer extract evidence telling us explicitly
which paths to follow in graphs such as those pictured
in Figures~\ref{fig:myhilliso55} and~\ref{fig:NinftyMPLPO}.

To do this in a fine-grained way\footnote{Maybe more fine-grained than is
wise.}, we are going to
take some inspiration
from~\cite{RSSmod,Kihara24rethinking,Swan24} and
use some modalities.

\begin{definition}
A (pointed) modality is a map $\modal : \Omega \to \Omega$ which is monotone
and inflationary; that is we should have:
  \[ \forall p \; q \in \Omega. \; \left( (p \Rightarrow q) \wedge \modal p \Longrightarrow \modal q \right)
  \qquad \text{and} \qquad \forall p \in \Omega. \; \left( p \Rightarrow \modal p \right)\]
We call a proposition $p$ $\modal$-stable if the converse
$\modal p \Rightarrow p$ holds and, for any set $X$, we say that $A \subseteq X$
is a $\modal$-subset of $X$ if, for every $x \in X$, the proposition $``x \in A''$ is $\modal$-stable.

Given two modalities $\modal, \modal' \in \Omega$, we say that $\modal \le \modal'$
if for every $p \in \Omega$, $\modal p$ implies $\modal' p$.
\end{definition}

The double-negation map $\dneg : \Omega \to \Omega$ is the main modality we are
going to use. In realizability models, double-negation erases the computational
content of propositions (that is, double-negated statement can always be trivially
witnessed). The other one, which we will use 
to make a couple positive results a bit sharper, is defined as follows.

\begin{definition}[The mindchange modality]
For any proposition $p \in \Omega$, define $\modal_{\LPO} p$ as
\[ \exists x \in \Ninfty. ~~ \left( \left(x \in \underline{\bN} \vee x = \infty\right) \Longrightarrow p\right)\]
\end{definition}
The basic intuition is that proving $\modal_{\LPO} p$ is the same as proving
$p$ after asking an $\LPO$ question. The witness $x \in \Ninfty$ indicates
that we want to ask whether $x = \underline{t}$ for some $t$ or whether $x =\infty$
before we can prove $t$.

\begin{lemma}
  \label{lem:LPOltdneg}
  We have $\modal_{\LPO} \le \neg\neg$.
\end{lemma}
\begin{proof}
Assume that $\neg p$ holds and $\modal_{\LPO} p$, that is, we have some $x \in \Ninfty$
  and $x \in \underline{\bN} \cup \{\infty\}$ implies $p$. Having $x \in \underline{\bN}$
leads to a contradiction. Thus we have $x \not\in \underline{\bN}$, so by induction we can
show that $x \ge \underline{n}$ for all $n \in \bN$, and thus that $x = \infty$.
But this is also a contradiction.
\end{proof}

\begin{remark}
\label{rem:modWei}
The nomenclature $\modal_{\LPO}$ is chosen because $\modal_{\LPO}$ is an instance of the following
construction: to a map $r : S \to Q$, one may assign the following endomap
on truth values:
\[
\begin{array}{llcl}
  \modal_r :& \Omega &\longto&\Omega\\
            &  p &\mapsto& \exists q \in Q. \; \left(\left(\exists s \in S. \;  r(s) = q\right) ~ \Longrightarrow ~p \right)
\end{array}
\]
Regarding $r$ as a problem with inputs $q \in Q$ and solutions $r^{-1}(q)$, proving
$\modal_r p$ means ``we can prove $p$ provided we can solve one instance of $r$''.
In particular ``$r$ has a section'' proves that $\modal_r$ is the identity over $\Omega$.

Taking $r$ in this construction to be the obvious map $\bN + \{ \infty\} \to \Ninfty$
gives $\modal_\LPO$, as in this instance $\exists s \in \bN + \{\infty\}. \;  r(s) = q$
is equivalent to the disjunction of
\[
\exists n \in \bN. \; r(\inl(n)) = q
\qquad \text{or} \qquad
r(\inr(\infty)) = q
\]
while we have $r(\inl(n)) = \underline{n}$ and $r(\inr(\infty)) = \infty$.
\end{remark}

Note that, unlike $\dneg$, $\modal_{\LPO}$ is \emph{not} idempotent: to prove
$\modal_{\LPO}\modal_{\LPO} p$, one is allowed to ask two $\LPO$ questions
in succession.
To get a closure operators, one would need to allow an arbitrary finite number
of mindchanges, resulting in a modality better named $\modal_{\LPO^\diamond}$ or
$\modal_{\mathsf{C}_\bN}$ (see~\cite[Proposition 10]{NeumannP18}). Let us
define it now.

\begin{definition}[The finite mindchange modality]
Define the modality $\modal_{\LPO^\diamond}$ as the map which takes $p \in \Omega$ to
\[ \exists s : \bN \to \bN. \quad \left(
  \text{$s$ is monotone, $\neg\left( \lim_{n \to +\infty} s(n) = +\infty\right)$ and }
\left(\forall \ell \in \bN. \; \lim_{n \to +\infty} s(n) = \ell ~ \Longrightarrow ~ p\right)\right)\]
\end{definition}

We now define the weakening of the (strong) Myhill property. In what follows,
$\modal$ is an arbitrary modality.

\begin{definition}
We say that a set $X$ has the $\modal$ strong Myhill property when, for
every pair of injections $f, g : X \to X$, there is a bijection
  $h : X \to X$ such that, for every $x \in X$,
$\modal\left(h(x) \in \bigcup_{m \in \bZ} f \circ (g \circ f)^{m}\right)$ holds.

We say that $X$ has the $\modal$ Myhill property if, for any $\modal$-subsets
$A, B \subseteq X$, we have that $A \preceq_1^X B$ and $B \preceq_1^X A$ imply
that $A \simeq^X B$.
\end{definition}

By unpacking the definitions and~\Cref{lem:LPOltdneg} one can
establish the following.

\begin{proposition}
Let $\modal \le \modal'$ be two modalities and $X$ be a set. The following holds:
\begin{itemize}
\item If $X$ has the $\modal$ (strong) Myhill property, then it also has the
    $\modal'$ (strong) Myhill property.
\item If $X$ has the $\modal$ strong Myhill property, then it has the $\modal$ Myhill property
\end{itemize}
\end{proposition}

To which extent those implications are strict is unclear, but we can note the
following when it comes to the $\neg\neg$ modality for certain sets and models.

\begin{remark}
  \label{rem:modest}
If $X$ is a modest set in a realizability model,
then if it holds that $X$ has the $\neg\neg$ Myhill property, then it necessarily has the
$\neg\neg$ strong Myhill property.
\end{remark}
\begin{proof}
Assume we have a realizer $e$ for $X$ having the $\neg\neg$ Myhill property.
Up to trivial modifications, it takes as input evidence for $A$ and $B$ being $\dneg$-stable,
codes for injections $f$ and $g$ together
with evidence that $f$ and $g$ reduce $A$ to $B$ and vice-versa and outputs a
code for a bijection $h$ such that $h(A) = B$.
Consider the families of sets
$(A_{f,g,x})_{f,g,x}$ and $(B_{f,g,x})_{f,g,x}$ indexed by injections $f, g : X \to X$
and $x \in X$ defined by
\[
  \begin{array}{l !\quad c !\quad l}
  z \in A_{f,g,x} &\Longleftrightarrow& \neg\neg z \in \bigcup_{m \in \bZ} (g \circ f)^{m}(x) \\
  y \in B_{f,g,x} &\Longleftrightarrow& \neg\neg y \in \bigcup_{m \in \bZ} \left(f \circ (g \circ f)^{m}\right)(x)
\end{array}
\]
Those sets are clearly $\neg\neg$-stable, and this is realized by a trivial
realizer $\bullet$ that does not depend on $f,g$ and $x$. Similarly, evidence that
$f$ and $g$ are reductions can also be realized by the same realizer $\bullet$.
So overall,
$e' := \lambda f. \lambda g. \; e \; \bullet \; \bullet \; f \; g \; \bullet \; \bullet$
realizes ``if $f, g : X \to X$ are injections reducing $A_{f,g,x}$ to $B_{f,g,x}$
and vice-versa, there is a bijection $h : X \to X$ such that $\dneg h(A_{f,g,x}) = B_{f,g,x}$'',
and it does so without inspecting the code of $x$. Furthermore, because we assume
that $X$ is modest, we know that for any codes $e_f$ and $e_g$ for injections,
$e' \; e_f \; e_g$ determines at most one possible bijection $X \cong X$
that must work for all pairs $A_{f,g,x}$ and $B_{f,g,x}$.
Therefore, it follows that $e'$ realizes that $X$ has the $\dneg$ strong Myhill property.
\end{proof}

This remark essentially applies to most concrete examples we examine in the
sequel, as modest sets are closed under exponentiation, (co)products, and
contain $1$, $\bN$ and $\Ninfty$.

\subsection{The conatural numbers have the $\modal_{\LPO}$ strong Myhill property}

\begin{theorem}
  \label{thm:ninftymyhill}
  Assuming $\MP$, $\mathbb{N}_\infty$ has the $\modal_{\LPO}$ strong Myhill property.
\end{theorem}

For this whole subsection, let us thus fix some injections
$f, g : \mathbb{N}_\infty \to \mathbb{N}_\infty$, from which we want to produce
some bijection $h \subseteq \bigcup_{m \in \bZ} f \circ (g \circ f)^{m}$.

The gist of the proof is that one can attempt to build $h$
over $\mathbb{N}_\infty \setminus \{\infty\}$ using the same back-and-forth
construction as in the classical Myhill isomorphism theorem if we
optimistically assume that $f(\infty) = g(\infty) = \infty$.
Of course this is not something we can simply assume nor check, but we can
check that $\inf(f(\infty), g(\infty))$ is ``large enough'' for a
given stage of the construction to not be wrong. If that quantity is always
large enough, this will mean that $f(\infty) = g(\infty) = \infty$ and the
construction will naturally enforce that $h(\infty) = \infty$.
Otherwise, at some stage we will realize (via~\Cref{lem:MPmin}) that either
$f(\infty) \neq \infty$ or $g(\infty) \neq \infty$. In either case, we
can apply~\Cref{lem:badinjectionLPO} to deduce that
$\Ninfty \cong \bN$, which allows us to continue the back-and-forth as if
$\Ninfty$ was actually $\bN$ much like in the original proof of the Myhill
isomorphism theorem.

Now let us make all of that more precise. First, let us define what are
the approximations we are going to build. There are going to be two kinds,
those that are continuous, corresponding to the first optimistic phase.

\begin{definition}
A \emph{continuous partial solution} is a finite subset $I$ of $\bN \times \bN$
such that:
\begin{itemize}
\item $I$ is a partial isomorphism: for all $n, n', m, m' \in \bN$, we have
\[
(n, m) \in I ~\text{and}~ (n', m') \in I \qquad \text{implies} \qquad n = n' \Longleftrightarrow m = m'\]
\item the image of $I$ by $(n,m) \mapsto (\underline{n},\underline{m})$ is included in $\bigcup\limits_{m \in \bZ} f \circ (g \circ f)^{m}$
\end{itemize}
We say $I$ has rank (at least) $n$ if
\begin{itemize}
\item for every $k < n$, $k$ belongs to both the domain and codomain of $I$
\item $\inf(f(\infty), g(\infty)) \ge \underline{n}$.
\end{itemize}
\end{definition}

The second kind corresponds to the phase where we noticed
$\inf(f(\infty), g(\infty)) \neq \infty$ and have $\LPO$ available.

\begin{definition}
A \emph{discontinuous partial solution} is a finite subset $I$ of $\Ninfty \times \Ninfty$
such that:
\begin{itemize}
\item $I$ is a partial isomorphism: for all $x, y', x, y' \in \Ninfty$, we have
\[
(x, y) \in I ~\text{and}~ (x', y') \in I \qquad \text{implies} \qquad x = x' \Longleftrightarrow y = y'\]
\item $I \subseteq \bigcup\limits_{m \in \bZ} f \circ (g \circ f)^{m}$
\item there exists $n, m \in \bN$ such that $(\underline{n}, \infty) \in I$ and $(\infty, \underline{m}) \in I$
\end{itemize}
Note in particular that the existence of a discontinuous partial solution implies
that $\LPO$ holds using the last criterion.

We say $I$ has rank (at least) $n > 0$ if 
\begin{itemize}
\item for every $k < n$, $\underline{k}$ belongs to both the domain and codomain of $I$
\end{itemize}
\end{definition}

We say that a partial solution $I$ extends a partial solution $J$ if we have
either $J \subseteq I$ or that $J$ is continuous, $I$ discontinuous and the
image of $J$ by $(n,m) \mapsto (\underline{n}, \underline{m})$ is included in
$I$. Now the first step towards proving~\Cref{thm:ninftymyhill} is to
establish the following.

\begin{lemma}
  \label{lem:ninftymyhillapprox}
Assuming Markov's principle, there exists a sequences of partial solutions
\[(I_n)_{n \in \bN} \in \left(\powerset(\bN \times \bN) + \powerset(\Ninfty \times \Ninfty)\right)^\bN\]
such that $I_n$ has rank at least $n$ and $I_n$ extends $I_m$ whenever $n \ge m$.
\end{lemma}

To prove this, it suffices to show that any partial solution of rank $n$
can be extended to a solution of rank $n+1$. For discontinuous partial solutions,
this is easy.

\begin{lemma}
  \label{lem:discontinuouspartialsolutionext}
There is a function that turns partial discontinuous solution of rank $n$ into
partial discontinuous solution of rank $n + 1$.
\end{lemma}
\begin{proof}
Since we are given a partial discontinuous solution, we know that $\LPO$ holds.
Thus we can simply use the isomorphism $\Ninfty \cong \bN$ to extend the solution
as in the proof of~\Cref{thm:myhilliso55}.
\end{proof}

The more difficult part is thus to extend partial continuous solutions. In the
presence of $\LPO$, we can straightforwardly use~\Cref{lem:discontinuouspartialsolutionext}
since we can then turn a partial continuous solution into a discontinuous one.

\begin{lemma}
There is a function that turns partial continuous solutions $I$ of rank $n$ into
a compatible partial discontinuous solutions $I'$ of rank $n$ if $\LPO$ holds.
\end{lemma}
\begin{proof}
First take the image of $I$ under $(n,m) \mapsto (\underline{n}, \underline{m})$;
that is almost a partial discontinuous solution, except for the fact that $\infty$
is not in the domain or codomain. Since we have $\LPO$, we can again use
the isomorphism $\Ninfty \cong \bN$ to make equality on $\Ninfty$ decidable and
add $\infty$ to the domain and codomain as in the proof of~\Cref{thm:myhilliso55}.
\end{proof}

With that in hand, it thus suffices to show the existence of a procedure that
takes a partial continuous solution $I_n$ of rank $n$ and returns either
\begin{itemize}
  \item a partial continuous solution of rank $n+1$
  \item or a distinct dummy value with the guarantee that $\LPO$ holds
\end{itemize}
This is done in two steps.
First, we check whether
$\inf(f(\infty),g(\infty)) \ge \underline{n + 1}$, a precondition to an extension
existing in the first place; if it's not, we return the dummy value.
This can be done because, by~\Cref{lem:MPmin} we can decide whether $f(\infty) \le \underline{n+1}$
or $\underline{n+1} \ge f(\infty)$, and similarly for $g$. If we are forced
to return the dummy value, we know that $f(\infty) \in \underline{\bN}$ or
$g(\infty) \in \underline{\bN}$, and we can deduce $\LPO$ from~\Cref{lem:inftyToIsolated}.
We can then proceed with the second step, which corresponds to a back-and-forth
similar to what is going on in the proof of~\Cref{thm:myhilliso55}: first we attempt
to add $n$ to the domain of $I$ if necessary, and then the codomain in a similar
manner. Since both steps are symmetrical, we focus on proving only the following
statement.

\begin{lemma}
\label{lem:forthCont}
There is a function $\Phi$ taking as input a partial continuous
solution $I$ and $n \in \bN$ such that:
\begin{itemize}
\item either $\Phi(I, n)$ is defined and returns some $m \in \bN$ such that
$I \cup (n, m)$ is a partial continuous solution
\item or a dummy value with the guarantee that $\LPO$ holds.
\end{itemize}
\end{lemma}
\begin{proof}
If $n \in \dom(I)$, this is trivial; so assume $n \not\in \dom(I)$.
To compute $\Phi(I, n)$, first we compute
$\inf(f(\underline{n}), f(\infty))$. By~\Cref{lem:MPmin} and the injectivity
of $f$, we know this value is in $\underline{\bN} \cong \bN$ and thus that there
is $m \in \bN$ such that
\[ \text{either } \quad f(\underline{n}) = \underline{m} \quad \text{ or } \quad f(\infty) = \underline{m}\]
In the latter case, we let $\Phi$ return the dummy value as we know $\LPO$ holds by~\Cref{lem:inftyToIsolated}.
Otherwise, we check if $m$ is in the codomain of $I$.
If it isn't, $\Phi(I, n) = m$. If it is, we then compute $\inf(g(\underline{m}), g(\infty))$
and similarly have some $n'$ such that
\[ \text{either } \quad g(\underline{m}) = \underline{n'} \quad \text{ or } \quad g(\infty) = \underline{n'}\]
Similarly as before, we can assume without loss of generality we are in the first case, as otherwise $\LPO$
holds. Then we recursively compute $\Phi(I \setminus \{(n',m)\},n')$. This process will always halt
as the size of $I$ decreases at each step. Since we follow $f$ and $g$, $I \cup (n, \Phi(I,n))$
is still a partial continuous solution.
\end{proof}

This statement, together with the previous remarks, thus concludes the proof
of~\Cref{lem:discontinuouspartialsolutionext}. All that is left is thus
to argue that this allows us to build a suitable bijection.

\begin{proof}[Proof of~\Cref{thm:ninftymyhill} from~\Cref{lem:discontinuouspartialsolutionext}]
Assume a suitable sequence of partial solutions $(I_n)_{n \in \mathbb{N}}$
and recall we defined $\Ninfty \subseteq \Cantor$ as the set of sequence with
a single $1$. For $x \in \Ninfty$ and $n \in \bN$, define $h(x)_n \in 2$ by cases:
\begin{itemize}
  \item If $I_{n+1}$ is continuous and $x = \underline{m}$ for any $(m,k) \in I_{n+1}$,
  we set $h(x)_n = \underline{k}_n$.
\item Otherwise, if $I_{n+1}$ is still continuous, we set $h(x)_n = 0$.
\item If $I_{n+1}$ is discontinuous, then $\LPO$ holds and we can check if $x$ is in the domain of $I_{n+1}$. If it
is and $(x, y) \in I_{n+1}$, we set $h(x)_n = y_n$, otherwise we set $h(x)_n = 0$.
\end{itemize}
It is relatively straightforward to check that this definition of $h$ is total and
that its inverse can be defined in a similar way.
Now we need to check that we have, for every $x \in \Ninfty$ that 
\[\modal_{\LPO}\left(\exists m \in \bZ. \; h(x) = \left(f \circ (g \circ f)^{m}\right)(x)\right)\]
To this end, fix such an $x \in \Ninfty$ and recall that the proposition above
unfolds to
  \[\exists w \in \Ninfty. ~~ \left(\left(w \in \underline{\bN} \vee w = \infty\right) \; \Longrightarrow \;
  \exists m \in \bZ. \; h(x) = \left(f \circ (g \circ f)^{m}\right)(x)\right)\]
Note we can compute an element $s \in \Ninfty$
such that $s \ge \underline{n}$ if and only if $I_n$ is continuous.
To witness the existential statement, we take $w = \inf(x, s)$.

Intuitively, we
first start with the belief that $x = s = \infty$, which entails $f(\infty) = g(\infty) = \infty$.
In such a case ($w = \infty$), we can take $m = 0$ and we know we are in the right.
Otherwise, we may possibly notice that there is $n \in \bN$ such that $w = \underline{n}$.
By~\Cref{lem:minwoMP}, this means that either $s = \underline{n}$ or $x = \underline{n}$,
two cases we may treat separately:
\begin{itemize}
  \item if $s = \underline{n}$, it means $\LPO$ holds. In such a case, we may
    decide whether $x \in \underline{\bN}$ or not.
    \begin{itemize}
      \item if $x = \infty$, then we may examine $I_n$ and replay the
        construction of $I_n$ to compute $m$.
      \item if $x = \underline{k}$ for some $k \in \bN$, then one may similarly
        examine the construction of $I_{k + 1}$ to compute $m$.
    \end{itemize}
  \item if $x = \underline{n}$, then it suffices to replay the construction of
    $I_{n+1}$ to compute $m$.
\end{itemize}
\end{proof}

\section{Negative results for the $\neg\neg$ Myhill property}
\label{sec:negnegative}

\subsection{Concrete examples for which the $\neg\neg$ Myhill property is $\LPO$-hard}

We now turn to arguing that the classes of sets with the Myhill property cannot
be closed under products, coproducts or exponentiations. We first give some
concrete sets for which having the $\dneg$ strong Myhill property implies $\LPO$. This
implies that it cannot be the case that sets with the Myhill property can be
closed under (co)products or exponentiation in constructive logic.

When using the assumption that some set has the $\dneg$ Myhill property, we will 
need to exhibit particular $\dneg$-stable subsets. To establish that the
subsets we want to use are $\dneg$-stable, we will always be able to check
that they are effectively closed, by which we mean they are inverse images of $\infty$
by some map\footnote{This matches the definition of synthetic topology provided we take Sierpi{\'n}ski space to be the coequalizer of the 
maps $\id, \Succinfty : \Ninfty \to \Ninfty$, as is typically done in computable analysis.}.

\begin{lemma}
For any set $X$ and $\chi : X \to \Ninfty$,
$\chi^{-1}(\infty)$ is $\dneg$-stable.
\end{lemma}
\begin{proof}
Abstracting away, this follows from $\neg x \neq \infty \Rightarrow x = \infty$; from the
premise one can prove $x \ge \underline{n}$ by induction over $n$ and thus
conclude.
\end{proof}

We will not explicitly appeal to this lemma and will leave to the reader
checking that the sets we are using do fit in this pattern. Let us give
a few examples of such constructions that cover the main ideas behind our
use-cases.

\begin{example}
For any $s \in \Baire$, $\{s\} = \chi_s^{-1}(\infty)$ for some $\chi_s$
defined by corecursion in the obvious way:
\[ \chi_{ns}(mp) = \left\{ \begin{array}{ll}
  \underline{0} & \text{if $n \neq m$} \\
  \Succinfty(\chi_s(p)) & \text{otherwise}
\\ \end{array} \right. \]
\end{example}

\begin{example}
For the set $2\Ninfty = \{ x + x \mid x \in \Ninfty\} \subseteq \Baire$,
we may define $\chi_{2\Ninfty} : \Baire \to \Ninfty$ such that $\chi_{2\Ninfty}^{-1}(2\Ninfty) = 2\Ninfty$
by corecursion as follows:
\[
\chi_{2\Ninfty}(1p) = \chi_\infty(p) \qquad
\chi_{2\Ninfty}(0(k+1)p) = \underline{0} \quad \text{and} \quad
\chi_2(00p) = \Succinfty(\Succinfty(\chi_2(p)))
\]
\end{example}

\begin{example}
If we are given $\chi_1, \chi_2 : X \to \Ninfty$, then the preimage of
$\infty$ by the pointwise sum $\chi_1 + \chi_2$ is the union
$\chi_1^{-1}(\infty) \cup \chi_2^{-1}(\infty)$.
\end{example}

\begin{theorem}
  \label{thm:notCoproddneg}
Assuming Markov's principle, the following are equivalent:
\begin{itemize}
\item the limited principle of omniscience
\item $\bN + \Ninfty$ has the $\neg\neg$ strong Myhill property
\end{itemize}
\end{theorem}
\begin{proof}
We focus on proving that $\bN + \Ninfty$ having the $\dneg$ Myhill property implies
$\LPO$; that $\LPO$ implies the other direction is straightforward.

Recall the functions $f, g : \bN \to \bN$ and their lifts
$\underline{f}, \underline{g} : \Ninfty \to \Ninfty$ defined in the proof
of~\Cref{thm:oscNinftyNotMP}. This allows us to argue that
the following functions $f', g' : 2 \times \Ninfty \to 2 \times \Ninfty$ are
well-defined
\[
  \begin{array}{llcl !\qquad llcl}
f' : & \bN + \Ninfty &\longto& \bN + \Ninfty
    &
g' : & \bN + \Ninfty &\longto& \bN + \Ninfty
\\
    & \inr(x) &\longmapsto& \inr(f(x))
    &
    & \inr(x) &\longmapsto& \inr(g(x))
\\
    & \inl(n) &\longmapsto& \inr(\underline{\tuple{2n,0}})
    &
    & \inl(n) &\longmapsto& \inr(\underline{\tuple{2n+1,0}})
\end{array}
\]
The corresponding graph is pictured in~\Cref{fig:notCoproddneg}; those
functions are easily checked to be injective.
Now,
assume we have a bijection $h : \bN + \Ninfty \to \bN + \Ninfty$
such that for every $x \in \bN + \Ninfty$,
$\dneg \exists m \in \bZ. \; h(x) = \left(f' \circ (g' \circ f')^m\right)(x)$.
Write $\alpha : \Ninfty \to 2$ the map defined by $\alpha(x) = 1 \; \Longleftrightarrow \;
\exists n \in \bN. \; h(\inr(x)) = \inl(n)$.

\begin{figure}[h]
\[\begin{tikzcd}[cramped,row sep=small]
	&&& {} \\
	& \vdots & {} & {} & \vdots \\
	& {\underline{\langle0,2\rangle}} &&& {\underline{\langle0,2\rangle}} \\
	& {\underline{\langle1,1\rangle}} &&& {\underline{\langle1,1\rangle}} \\
	& {\underline{\langle2,0\rangle}} &&& {\underline{\langle2,0\rangle}} \\
	& {\underline{\langle0,1\rangle}} &&& {\underline{\langle0,1\rangle}} \\
	& {\underline{\langle1,0\rangle}} &&& {\underline{\langle1,0\rangle}} \\
	& {\underline{\langle0,0\rangle}} &&& {\underline{\langle0,0\rangle}} \\
	& 0 &&& 0 \\
	& 1 &&& 1 \\
	& \vdots &&& \vdots \\
	f' & {\mathbb{N} \uplus \Ninfty} &&& {\mathbb{N} \uplus \Ninfty} & g'
	\arrow[color={rgb,255:red,153;green,92;blue,214}, dashed, no head, from=3-2, to=2-3]
	\arrow[color={rgb,255:red,153;green,92;blue,214}, from=3-5, to=3-2]
	\arrow[color={rgb,255:red,168;green,31;blue,0}, from=4-2, to=4-5]
	\arrow[color={rgb,255:red,168;green,31;blue,0}, dashed, no head, from=4-5, to=2-4]
	\arrow[color={rgb,255:red,23;green,94;blue,23}, dashed, no head, from=5-2, to=1-4]
	\arrow[color={rgb,255:red,23;green,94;blue,23}, from=5-5, to=5-2]
	\arrow[color={rgb,255:red,153;green,92;blue,214}, from=6-2, to=3-5]
	\arrow[color={rgb,255:red,153;green,92;blue,214}, from=6-5, to=6-2]
	\arrow[color={rgb,255:red,168;green,31;blue,0}, from=7-2, to=7-5]
	\arrow[color={rgb,255:red,168;green,31;blue,0}, from=7-5, to=4-2]
	\arrow[color={rgb,255:red,153;green,92;blue,214}, from=8-2, to=6-5]
	\arrow[color={rgb,255:red,153;green,92;blue,214}, from=8-5, to=8-2]
	\arrow[color={rgb,255:red,153;green,92;blue,214}, squiggly, from=9-2, to=8-5]
	\arrow[color={rgb,255:red,168;green,31;blue,0}, squiggly, from=9-5, to=7-2]
	\arrow[color={rgb,255:red,23;green,94;blue,23}, squiggly, from=10-2, to=5-5]
	\arrow[color={rgb,255:red,254;green,156;blue,52}, squiggly, no head, from=10-5, to=2-3]
	\arrow["{{{:}}}"{description}, draw=none, from=12-1, to=12-2]
	\arrow[shift right, shorten <=8pt, shorten >=8pt, from=12-2, to=12-5]
	\arrow[shift right, shorten <=8pt, shorten >=8pt, from=12-5, to=12-2]
	\arrow["{{{:}}}"{description}, draw=none, from=12-6, to=12-5]
\end{tikzcd}\]
\caption{Injections used for~\Cref{thm:notCoproddneg}. The squiggly arrows force the bijection to alternate between the $\bN$ and the $\Ninfty$
components. The coprojections $\inl$ and $\inr$ have been suppressed from the
picture for readability.}
\label{fig:notCoproddneg}
\end{figure}
By~\Cref{lem:oscillation}, it suffices to
show that $\alpha$ takes both values $0$ and $1$ infinitely often. To this
end, fix $n \in \bN$ and let us show that for each $i \in 2$
we can find $k_i \ge n$ such that $\alpha(\underline{k_i}) = i$:
\begin{itemize}
\item First we can show by induction that $\dneg \exists m \in \bN. \; h^{-1}(\inl(n)) = \inr(\underline{\tuple{2n+1,m}})$.
Since $\tuple{2n+1, m} \ge n$ for every $m \in \bN$, this means a fortiori that
it is not the case that for every $k \ge n$ that $\alpha(\underline{k}) = 0$. By Markov's principle
it means there exists some $k_1$ such that $\alpha(\underline{k_1}) = 1$.
\item
For the other case, 
first note that
for any $k, \ell \in \bN$,
\[h(\inr(\underline{\tuple{2n + 1, k}})) = \inl(\ell) \quad \Longrightarrow \quad \ell = n\]
$\neg\neg\; \ell = n$ can be proven as follows: assume without loss of generality
that there is some $m \in \bZ$ such that $h(\underline{\tuple{2n+1}, k}) = (f' \circ (g' \circ f')^m)(\underline{\tuple{2n+1, k}})$, and run an induction
over $|m|$. Then use the fact that equality on $\bN$ is $\dneg$-stable to deduce
$\ell = n$.

With that in mind, we can check the value of $\alpha(\underline{\tuple{2n+1,0}})$.
If it is $0$, we can take $k_0 = \tuple{2n+1,0}$; otherwise, we
have $h(\inr(\underline{\tuple{2n + 1, 0}})) = \inl(n)$, and since $h$ is injective,
we cannot have $h(\inr(\underline{\tuple{2n+1,1}})) = \inl(n)$. By our remark,
it cannot be any $\inl(\ell)$, so $\alpha(\underline{\tuple{2n+1,1}}) = 0$
and we can take $k_0 = \tuple{2n+1,1}$.
\end{itemize}
\end{proof}

\begin{theorem}
  \label{thm:nninftynotmyhill}
Assuming Markov's principle, the following are equivalent:
\begin{itemize}
\item the limited principle of omniscience
\item $\bN \times \Ninfty$ has the $\neg\neg$ Myhill property
\item $\Ninfty^2$ had the $\dneg$ Myhill property.
\end{itemize}
\end{theorem}
\begin{proof}
That $\LPO$ implies those sets have the ($\neg\neg$) Myhill property
is because 
$\LPO$ implies $\Ninfty \cong \bN$ and thus $\bN \times \Ninfty \cong \Ninfty^2 \cong \bN$.

For the other directions, let's first work with $\bN \times \Ninfty$.
Take $A = \{(0, \underline{0})\} \cup \bN \times \{\infty\}$
and $B = \bN \times \{\infty\}$. That $A$ and $B$ are $\dneg$-stable follows
from Markov's principle (essentially~\Cref{lem:NinftyMP} and routine checks).
Define $f : \bN \times \Ninfty \to \bN \times \Ninfty$
by setting $f(0,\underline{0}) = (0, \infty)$ and $f(n, x) = (n + 1, x)$.
$(0,\underline{0})$ is isolated, so this defines a total function which witnesses
that $A \preceq_1^{\bN \times \Ninfty} B$. Similarly, $g : (n, x) \mapsto (n, x + \underline{1})$
witnesses that $B \preceq_1^{\bN \times \Ninfty} A$. Since we are assuming that
$\bN \times \Ninfty$ has the Myhill property, we have a bijection $h : \bN \times \Ninfty \to \bN \times \Ninfty$
such that $h(0, \underline{0}) = (k, \infty)$ for some $k \in \bN$. Then
$x \mapsto h^{-1}(k,x)$ is an injection $\Ninfty \to \bN \times \Ninfty$ sending
$\infty$ to an isolated point, so we can conclude by~\Cref{lem:inftyToIsolated}.

For $\Ninfty^2$, we can use the exact same idea with
$A = \{(\underline{0}, \underline{0})\} \cup \underline{\bN} \times \{\infty\}$
and $B = \underline{\bN} \times \{\infty\}$.
\end{proof}

Because of~\Cref{thm:ninftymyhill} (and Myhill's isomorphism theorem), either
of these theorems
shows that Markov's principle and closure under products of the Myhill property
imply $\LPO$.
We can use this result to prove a similar result about Baire space and Cantor
space not constructively having the (strong) Myhill property.

\begin{theorem}
  \label{thm:bairenotstrongmyhill}
  If Markov's principle holds and either of $\Baire$ or $\Cantor$ has the $\neg\neg$ strong Myhill property, then $\LPO$ holds.
\end{theorem}

Before proving that, let us establish a preliminary lemma about extending
continuous\footnote{Here we assume the standard topologies on $\Baire$, $\Cantor$ induced
by the usual ultrametric and $\bN$ to be discrete. $\Ninfty$'s topology is simply inherited from $\Baire$.}
injections $\bN \times \Ninfty \to \bN \times \Ninfty$ to injections $\Baire \to \Baire$.
To simplify notations, write $\bN\Ninfty$ for the image of
$\bN \times \Ninfty$ in $\Baire$ by the map $(n, p) \mapsto np$; it is clear
that $\bN\Ninfty \cong \bN \times \Ninfty$.

\begin{lemma}
  \label{lem:contNNinftyext}
Assuming Markov's principle, if $f : \bN\Ninfty \to \bN\Ninfty$ has a modulus of continuity
$\mu : \bN \times \bN \to \bN$ such that $f(x)_n$ only depends on the $\mu(x_0, n)$ first
elements of $x$, then it extends to
a continuous injection $\overline{f} : \Baire \to \Baire$ such that
$\overline{f}^{-1}(\bN\Ninfty) = \bN\Ninfty$.
\end{lemma}
\begin{proof}
Morally, this can be done because $\bN \times \Ninfty$ is a nice enough closed
subspace of $\Baire$ such that, for any open $U$ of $\Baire$ with $U \cap \bN\Ninfty$
inhabited, we can inject $\Baire \setminus \bN\Ninfty$ in $U \setminus \bN\Ninfty$.

Let us require, without loss of generality, that $\mu$ is strictly
increasing on its second component.
We set $\overline{f}(x)_n = f(x)_n$ as long as the $\mu(n)$
first bits of $x$ are valid prefixes of elements of $\bN\Ninfty$. If for some minimal $n_0 > 0$ we have that
the first $\mu(x_0, n_0)$ bits of $x$ do not form a valid prefix of an element of $\bN\Ninfty$, for every $n \ge n_0$
we set $\overline{f}(x)_{n} = 2 + x_{n - n_0}$. $\overline{f}$ is well-defined by
recursion and also uniformly continuous, with the same modulus of continuity.
It is also easy to argue by induction that $\overline{f}(x) \in \bN\Ninfty$ if and only if $x \in \bN\Ninfty$.

All that remains to be shown is that $\overline{f}$ is injective.

Let us note two things before proceeding:
\begin{enumerate}
  \item \label{enumitem:xgowrong} Assuming $x \in \Baire$
if there is a minimal prefix of $x$ of length $n \in \bN$ that
is not valid prefix of an element of $\bN\Ninfty$, then $x_k = \overline{f}(x)_{n_0 + k} - 2$ for $n_0 = \min \{ \ell \in \bN \mid n \le \mu(x,\ell)\}$.
\item Using Markov's principle and the injectivity of $f$, we can define a
function $\nu : \bN \to \bN$ such that, for every $x, y \in \bN\Ninfty$,
$x_n \neq y_n$ implies that
the $\nu(n)$-long prefixes of $f(x)$ and $f(y)$ differ.
\end{enumerate}

With that in mind, fix some $x, y \in \Baire$
such that $\overline{f}(x) = \overline{f}(y)$.
First, we need to show that $x_0 = y_0$.
We can check that $x$ and $y$ are valid prefixes of some elements
of $\Ninfty$ up to $M_0 = \max(\mu(x_0,\nu(0)), \mu(y_0, \nu(0)))$.
If they aren't then we have $x = y$ by~\Cref{enumitem:xgowrong}.
If they are, then we have some $x'$ and $y' \in \Ninfty$
that match $x$ and $y$ respectively on their first $M_0$ bits such that
\[f(x')_k = \overline{f}(x)_k = \overline{f}(y)_k = f(y')_k \qquad \text{for $k \le \nu(0)$}\]
By definition of $\nu(0)$, we have $x'_0 = y'_0$ and a fortiori $x_0 = x'_0 = y'_0 = y_0$.

We now proceed to show that $x_n = y_n$ for every $n > 0$ by strong
induction; the argument is essentially the same.
Set $M_n = \mu(x_0, \nu(n))$. Recall that $x_0 = y_0$, and note that
$M_n \ge n$: since $f$ is injective, $\nu(n) \ge n$, and since $\mu$ is
strictly increasing, we have $M_n \ge \nu(n) \ge n$.
We can check that $x$ and $y$ are enumerare valid prefixes of some elements
of $\bN\Ninfty$ up to $M_n$. If they don't, we have $x = y$ by~\Cref{enumitem:xgowrong}.
If they do, then we have some $x'$ and $y' \in \Ninfty$
whose $M_n$-long prefixes match $x$ and $y$ respectively
\[f(x')_k = \overline{f}(x)_k = \overline{f}(y)_k = f(y')_k \qquad \text{for $k \le \nu(n)$}\]
So in particular, we have $x'_n = y'_n$ by definition of $\nu$. And since $M_n \ge n$,
we a fortiori have $x_n = x'_n = y'_n = y_n$.
\end{proof}

Similarly, we can regard $\Ninfty^2$ as a subspace of $\Cantor$ via the
following embedding and get an analogous extension lemma.
\[
\begin{array}{llcl}
j :& \Ninfty^2 &\longrightarrow& \Cantor \\
  & (p, q) &\longmapsto& \left\{\begin{array}{lcl}
       2n &\mapsto& p_n \\
       2n+1 &\mapsto& q_n \\
  \end{array}
       \right.
\end{array}
\]

\begin{lemma}
  \label{lem:Ninfty2-ext}
Assuming Markov's principle, if $f : \Ninfty^2 \to \Ninfty^2$ is a uniformly continuous injection, then
there is an explicit continuous injection $\overline{f} : \Cantor \to \Cantor$
such that $\overline{f} \circ j = j \circ f$ and
  $\overline{f}^{-1}(j(\Ninfty^2)) = j(\Ninfty^2)$.
\end{lemma}
\begin{proof}
Same idea as in the proof of~\Cref{lem:contNNinftyext}; left to the reader.
\end{proof}

\begin{proof}[Proof of~\Cref{thm:bairenotstrongmyhill}]
Let us treat the case of $\Baire$.
Take $A, B, f$ and $g$ as in the first stage of the proof of~\Cref{thm:nninftynotmyhill} modulo
the isomorphism $\bN \times \Ninfty \cong \bN\Ninfty$. $f$ and $g$ are
clearly effectively continuous. Hence, by~\Cref{lem:contNNinftyext}, we have injections
$\overline{f}, \overline{g} : \Baire \to \Baire$ that extend $f$ and $g$
respectively and such that $\overline{f}^{-1}(\bN\Ninfty) = \overline{g}^{-1}(\bN\Ninfty) = \bN\Ninfty$.
Using the assumption that $\Baire$ has the $\dneg$ strong Myhill property,
call $\overline{h}$
the bijection obtained from $\overline{f}$ and $\overline{g}$. For every $x \in \bN\Ninfty$, we thus have
\[ \dneg\quad \left( \exists {m \in \bZ}. \quad
  \overline{h}(x) = (\overline{f} \circ (\overline{g} \circ \overline{f})^m)(x)
=  ({f} \circ ({g} \circ {f})^m)(x)\right)
\]
By induction over $|m|$, we can deduce further that
$\dneg \overline{h}(x) \in \bN\Ninfty$; then we can use Markov's principle
to show that $\bN\Ninfty$ is a $\dneg$-subset of $\Baire$ and deduce that
$\overline{h}$ restricts to a map $h : \bN\Ninfty \to \bN\Ninfty$ and
$h(A) = B$. We can then conclude via the same argument
used to prove~\Cref{thm:nninftynotmyhill}.

The case of $\Cantor$ having the $\dneg$ strong Myhill property is treated similarly by
  using~\Cref{lem:Ninfty2-ext} instead of \Cref{lem:contNNinftyext}.
\end{proof}

\begin{remark}
  \label{rem:bairecantornotmyhill}
While~\Cref{thm:bairenotstrongmyhill} only mentions the $\dneg$ strong Myhill property,
it is sufficient to show that $\Cantor$ and $\Baire$ having the $\dneg$ Myhill property
is not provable constructively by~\Cref{rem:modest}:
the Kleene-Vesley realizability model invalidates that since $\Cantor$ and
$\Baire$ are modest sets in that setting.
\end{remark}
\subsection{An edge case: $2 \times \Ninfty$}

In the previous subsection, counter-examples for particular sets $X$
having the $\dneg$ Myhill property are obtained by
constructing (computable) injections $f$ and $g$ such that any corresponding
$h$ cannot be continuous. Here we look at $2 \times \Ninfty$ where such a strategy
cannot work as a continuous $h$ always exist. However, 
one cannot prove that it has the $\dneg$ Myhill property because a (potentially
multivalued) functional that produces $h$ from (representations of) $f$ and $g$
cannot be continuous.

First, let us show that continuous solutions always exist, after characterizing
continuous injections
$2 \times \Ninfty \to 2 \times \Ninfty$ (an illustration of the next lemma
is given in~\Cref{fig:continj2Ninfty}).

\begin{figure}[h]
\begin{center}
  \includegraphics[scale=0.7]{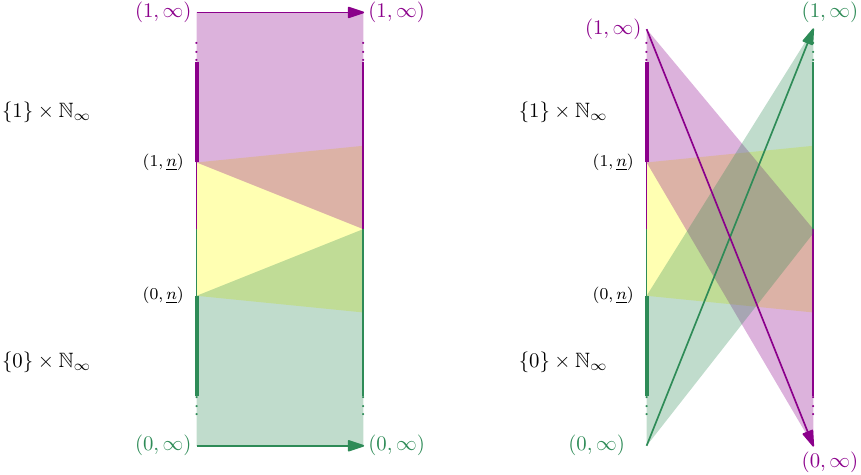}
\end{center}
\caption{Illustration of the two kinds of continuous injections
$f : 2 \times \Ninfty \to 2 \times \Ninfty$ according to whether
$f(0, \infty) = (0, \infty)$ (untangled case) or $(1, \infty)$ (tangled case).
The meaning of the three shaded areas is that an edge in the graph of $f$ must
always stay in one of the them.
In~\Cref{thm:contbij2Ninfty}, one can build a continuous bijection from two
such continuous injections $f, g$ of the same type by usual Cantor-Bernstein-style
arguments. For the other cases, it is useful to note that if both $f$ and $g$
are tangled, then $f \circ g$ is untangled and still a continuous injection.
}
\label{fig:continj2Ninfty}
\end{figure}
\begin{lemma}
  \label{lem:continjNNinfty}
Assuming classical logic, an injection
$f : 2 \times \Ninfty \to 2 \times \Ninfty$ is continuous if and only if
there is some $(j, n) \in 2 \times \bN$ and a continuous injection
$f' : 2 \times \Ninfty \to \Ninfty$ such that $f(i, \underline{n} + x) = (i + j \mod 2 ,f'(i,x))$.
\end{lemma}
\begin{proof}
The right-to-left direction is immediate.
For the left-to-right direction, first note that $(0, \infty)$
and $(1, \infty)$ are the only non-isolated points.
For $f$ to be injective, it needs to preserve non-isolated points, so
we necessarily have $f(0, \infty) = (j, \infty)$ and $f(1, \infty) = (1 - j, \infty)$
for some $j \in 2$.

By compactness of $\Ninfty$, there is some $n$ large enough such that
the image of $\{ x \in \Ninfty \mid x \ge \underline{n}\}$ by the map
\[
  \begin{array}{lcl}
    \Ninfty &\longto& 2 \times 2\\
    x &\longmapsto& (\pi_1(f(0,x)), \pi_2(f(1,x)))
  \end{array}
\]
is constantly $(j, 1 - j)$; we can then simply set $f'(i,x) = \pi_2(f(i,\underline{n} + x))$.
\end{proof}

\begin{lemma}
  \label{lem:NinftyNinftC0classical}
Assuming classical logic, an injection $f : \Ninfty \to \Ninfty$ is continuous
if and only if $f(\infty) = \infty$.
\end{lemma}
\begin{proof}
The topology on $\Ninfty$ is generated by sets $\{\underline{n}\}$
and $U_n = \{ x \in \Ninfty \mid x > \underline{n}\}$ when $n \in \bN$. It is obvious
that $f^{-1}(\{\underline{n}\})$ is open since $f$ is assumed to be injective,
so it simply remains to show that $f^{-1}(U_n)$ is open. Since $f$ is injective,
there can only be finitely values mapped outside of $U_n$, say all strictly inferior to
$\underline{k}$. So $f^{-1}(U_n)$ is $U_k$ joined with the finitely many values
strictly below $\underline{k}$, hence it is open.
\end{proof}

\begin{theorem}
  \label{thm:contbij2Ninfty}
Assuming classical logic, for any two continuous 
injections $f, g : 2 \times \Ninfty \to 2 \times \Ninfty$,
there exists a continuous bijection $h : 2 \times \Ninfty \to 2 \times \Ninfty$
  such that $\smash{h \subseteq \bigcup\limits_{m \in \bZ} f \circ (g \circ f)^m}$.
\end{theorem}
\begin{proof}
First let us treat the case where $f(0,\infty) = g(0, \infty) = (j, \infty)$
  for some $j \in 2$ (so that $f$ and $g$ are both ``tangled'' or both ``untangled'').
Using a suitable variant of the Cantor-Bernstein theorem~\cite{Banaschewski-1986},
there exists a bijection $h \subseteq f \cup g^{-1}$.
By~\Cref{lem:continjNNinfty}, we also have $n \in \bN$ large enough so that there exist
$f', g' : 2 \times \Ninfty \to \Ninfty$ such that
$f(i, \underline{n} + x) = (i + j \mod 2, f'(i, x))$ and
$g(i, \underline{n} + x) = (i + j \mod 2, g'(i,x))$ for every $x \in \Ninfty$.
Since $g$ is injective, there exists some $m \in \bN$ such that,
for any $i \in 2$, if $g(i, x) = (i', y)$ with $y \ge \underline{n + m}$,
then $x \ge \underline{n}$ (and therefore, $i = i'$).
Now we define $h' : 2 \times \Ninfty \to \Ninfty$ by
\[h'(i, x) = \left\{ \begin{array}{ll}
f'(i,x) & \text{if $h(i, \underline{n + m} + x) = f(i, \underline{n + m} + x) = (i + j \mod 2, f'(\underline{m} + i, x))$} \\
y - n& \text{otherwise (in which case, $h(i, \underline{n + m} + x)) = g^{-1}(i, \underline{n + m} + x) = (i, y))$)}
\end{array} \right.
\]
Due to the way we picked $n, m \in \bN$, we have 
$h(i, \underline{n + m} + x) = (i + j \mod 2, h'(i, x))$.
We also have that $h'(i,x) = \infty \Leftrightarrow x = \infty$, so $h'$ is
continuous by~\Cref{lem:NinftyNinftC0classical}, and by~\Cref{lem:continjNNinfty}
$h$ is continuous.

Now if $f(0, \infty) = (j, \infty)$ and $g(0, \infty) = (1 - j, \infty)$ (i.e.
  one of the elements of $\{f,g\}$ being ``tangled'' and the other ``untangled''),
we have $(g \circ f \circ g)(0, \infty) = (j, \infty)$ and $g \circ f \circ g$
remains injective. So if we use $g \circ f \circ g$ instead of $g$ in our
previous argument, we obtain a continuous
  $\smash{h \subseteq \bigcup\limits_{m \in \bZ} f \circ (g \circ f \circ g \circ f)^m}$, which
is enough as $f \circ (g \circ f \circ g \circ f)^m = f \circ (g \circ f)^{2m}$.
\end{proof}

This has for direct consequence that a doubly-negated version of the $\dneg$ strong Myhill
property holds in Kleene-Vesley realizability.

\begin{corollary}
\label{cor:dneg2NinftyMP}
The following is valid in Kleene-Vesley realizability:
given two injections $f, g : 2 \times \Ninfty \to 2 \times \Ninfty$, there $\dneg$-exists
a bijection $h : 2 \times \Ninfty \cong 2 \times \Ninfty$ such that
\[\forall x \in 2 \times \Ninfty. \;
  \neg\neg \left(h(x) \in \bigcup_{m \in \bZ} \left(f \circ (g \circ f)^m\right)(x)\right)\]
\end{corollary}

We now turn to showing that $2 \times \Ninfty$ does \emph{not} have the $\dneg$
Myhill property in Kleene-Vesley realizability. The key idea is that, while any continuous
$f$ and $g$ admit a continuous solution, that solution cannot be computed
continuously from codes for $f$ and $g$ in the setting of type 2
computability\footnote{Specifically, a code for a partial function
$f : \Baire \partto \bN$ is given by a tree whose leaves are labeled by outputs
in $\bN$. $f(p)$ is then defined if $p$ is a path leading to a leaf labeled
by $n$, see~\cite[\S 1.4.2]{VanOosten}.}.

\begin{theorem}
  \label{thm:noNNinftyMP}
Assuming classical logic, there is no continuous functional that takes codes for two continuous injections $f, g : 2 \times \Ninfty \to 2 \times \Ninfty$
to a code for a bijection $h : 2 \times \Ninfty \to 2 \times \Ninfty$ and its inverse so that $h \subseteq \bigcup_{m \in \bZ} f \circ (g \circ f)^m$.
\end{theorem}

The argument is illustrated in~\Cref{fig:noNNinftyMP}.
\begin{figure}
\[\begin{tikzcd}[cramped, column sep=small]
	&&&&&&& \vdots && \vdots && \vdots && \vdots \\
	&&&&&&& {\underline{5}} && {\underline{5}} && {\underline{5}} && {\underline{5}} \\
	&&& \vdots && \vdots & {} & {{{\underline{4}}}} && {\underline{4}} && {{{\underline{4}}}} && {\underline{4}} & {} \\
	&&& {\underline{3}} && {\underline{3}} & {} & {\underline{3}} && {\underline{3}} && {\underline{3}} &&
    \textbf{\textcolor{red}{{\underline{3}}}} & {} \\
	&& {\phantom{\vdots}} & {\underline{2}} && {\underline{2}} &&
    \textbf{\textcolor{blue}{{\underline{2}}}} && {\underline{2}} &&
    \textbf{\textcolor{blue}{{\underline{2}}}} && {\underline{2}} & {\phantom{\vdots}} \\
	{\{1\}\times\mathbb{N}_\infty} & \vdots & \vdots & {\underline{1}} && {\underline{1}} &&
    \textbf{\textcolor{blue}{{\underline{1}}}} &&
    \textbf{\textcolor{red}{{\underline{1}}}} &&
    \textbf{\textcolor{blue}{{\underline{1}}}} &&
    \textbf{\textcolor{red}{{\underline{1}}}} & {\phantom{\vdots}} & {\phantom{\mathbb{N}_\infty}} \\
	& {\underline{0}} & {\underline{0}} & {\underline{0}} && {\underline{0}} &&
    \textbf{\textcolor{blue}{{\underline{0}}}} &&
    \textbf{\textcolor{red}{{\underline{0}}}} &&
    \textbf{\textcolor{blue}{{\underline{0}}}} &&
    \textbf{\textcolor{red}{{\underline{0}}}} \\
	& \underline0 & \underline0 & \underline0 && \underline0 && \underline0 && \underline0 && \underline0 && \underline0 \\
	{\{0\}\times\mathbb{N}_\infty} & \vdots & \vdots & \vdots && \vdots && \vdots && \vdots && \vdots && \vdots && {\phantom{\mathbb{N}_\infty}} \\
	& {} & {} & {} && {} && {} && {} & \textcolor{rgb,255:red,214;green,0;blue,14}{{\text{or}}} & {} && {}
	\arrow[from=2-8, to=2-10]
	\arrow[dashed, no head, from=2-10, to=1-8]
	\arrow[from=2-12, to=2-14]
	\arrow[dashed, no head, from=2-14, to=1-12]
	\arrow[draw=none, from=3-4, to=4-6]
	\arrow[""{name=0, anchor=center, inner sep=0}, draw=none, from=3-7, to=4-7]
	\arrow[from=3-8, to=3-10]
	\arrow[from=3-10, to=2-8]
	\arrow[from=3-12, to=4-14]
	\arrow[from=3-14, to=2-12]
	\arrow[""{name=1, anchor=center, inner sep=0}, draw=none, from=3-15, to=4-15]
	\arrow[from=4-4, to=4-6]
	\arrow[from=4-4, to=5-6]
	\arrow[from=4-8, to=5-10]
	\arrow[from=4-10, to=4-8]
	\arrow[from=4-12, to=5-14]
	\arrow[from=4-14, to=4-12]
	\arrow[""{name=2, anchor=center, inner sep=0}, draw=none, from=5-3, to=6-3]
	\arrow[from=5-4, to=6-6]
	\arrow[from=5-6, to=5-4]
	\arrow[from=5-8, to=6-10]
	\arrow[from=5-10, to=5-8]
	\arrow[from=5-12, to=6-14]
	\arrow[from=5-14, to=5-12]
	\arrow[""{name=3, anchor=center, inner sep=0}, draw=none, from=5-15, to=6-15]
	\arrow[""{name=4, anchor=center, inner sep=0}, draw=none, from=6-1, to=9-1]
	\arrow[from=6-4, to=7-6]
	\arrow[from=6-6, to=6-4]
	\arrow[from=6-8, to=7-10]
	\arrow[from=6-10, to=6-8]
	\arrow[from=6-12, to=7-14]
	\arrow[from=6-14, to=6-12]
	\arrow[""{name=5, anchor=center, inner sep=0}, draw=none, from=6-16, to=9-16]
	\arrow[from=7-2, to=8-3]
	\arrow[from=7-3, to=7-2]
	\arrow[from=7-4, to=8-6]
	\arrow[from=7-6, to=7-4]
	\arrow[from=7-8, to=8-10]
	\arrow[from=7-10, to=7-8]
	\arrow[from=7-12, to=8-14]
	\arrow[from=7-14, to=7-12]
	\arrow[draw=none, from=8-2, to=7-2]
	\arrow[dashed, no head, from=8-2, to=9-3]
	\arrow[from=8-3, to=8-2]
	\arrow[draw=none, from=8-4, to=7-4]
	\arrow[dashed, no head, from=8-4, to=9-6]
	\arrow[from=8-6, to=8-4]
	\arrow[draw=none, from=8-8, to=7-8]
	\arrow[dashed, no head, from=8-8, to=9-10]
	\arrow[from=8-10, to=8-8]
	\arrow[draw=none, from=8-12, to=7-12]
	\arrow[dashed, no head, from=8-12, to=9-14]
	\arrow[from=8-14, to=8-12]
	\arrow["{{{\text{Stage 0}}}}"{description}, draw=none, from=10-2, to=10-3]
	\arrow["{{{\text{Stage 1}}}}"{description}, draw=none, from=10-4, to=10-6]
	\arrow["{{{\text{Stage 2a}}}}"{description}, draw=none, from=10-8, to=10-10]
	\arrow["{{{\text{Stage 2b}}}}"{description}, draw=none, from=10-12, to=10-14]
	\arrow["{{{\text{$(h, h^{-1})$ defined below $\squiggle$}}}}"{description}, equals, squiggly, from=0, to=1]
	\arrow["{{{\text{$h$ no longer goes underground}}}}"{description}, squiggly, no head, from=2, to=3]
	\arrow["{{{\text{ground level}}}}"{description, pos=0.1}, equals, from=4, to=5]
\end{tikzcd}\]

\caption{Informal illustration of the proof
of~\Cref{thm:noNNinftyMP} where $h$ designates partial outputs from the
would-be procedure that should produce a bijection compatible with $f$ and $g$.
The first components of elements of $2 \times \Ninfty$ are omitted,
with the convention that elements of $\{0\} \times \Ninfty$ sit in the bottom part
of the picture. First, through stage 0 and 1, we feed information consistent with
$(f_\infty,g_\infty)$ until we find some $n \in \bN$ such that
$\pi_1(h(0,y)) = \pi_1(h^{-1}(0,y)) = 0$ for any $y \ge \underline{n}$
(in this picture, $n = 2$).
Then we find $m > n$ large enough so that the values of $h$ and $h^{-1}$
over $A = \{1\} \times \{\underline{0}, \ldots, \underline{n-1}\}$ are the same
for any $(f_x, g_x)$ with $x \ge \underline{m}$ and those values do
not exceed $(1, \underline{m})$ (in the picture, $m = 3$).
Then according to whether
$|h(A) \cap \{1\} \times\{\underline{n}, \ldots, \underline{m}\}|$ and
$|h^{-1}(A) \cap \{1\} \times\{\underline{n}, \ldots, \underline{m}\}|$
are different or the same, we move to stage 2a or 2b respectively and derive a contradiction by
the finite pigeonhole principle. The elements of
$h(A) \cap \{1\} \times\{\underline{n}, \ldots, \underline{m}\}$ and
$h^{-1}(A) \cap \{1\} \times\{\underline{n}, \ldots, \underline{m}\}$ are respectively coloured
blue and red at stages 2a and 2b in the picture above.
}
\label{fig:noNNinftyMP}
\end{figure}

\begin{proof}
In this proof, we will only consider a family of pairs of continuous injections $f_x, g_x : 2 \times \Ninfty \to 2 \times \Ninfty$ indexed by $x \in \Ninfty$. Intuitively, $f_x$ and $g_x$ are defined so that the induced bipartite graph will have a single connected component if $x = \infty$ or be broken at a specific step if $x \in \underline{\bN}$ (see also~\Cref{fig:brokenLadder} for a picture). We may define those maps formally by
\[
\begin{array}{llcl !\quad llcl}
f_x : & 2 \times \Ninfty& \longto & 2 \times \Ninfty &
g_x : & 2 \times \Ninfty& \longto & 2 \times \Ninfty \\
& (1, \underline{0}) &\longmapsto & (0, \underline{0})
\\
& (0, y) &\longmapsto & (0, y + \underline{1}))
&
& (0, y) &\longmapsto & (0, y)\\
& (1, y + \underline{1}) &\longmapsto& 
(1, y)  ~~\text{if $y + y + \underline{1} < x$}
&
&(1, y) &\longmapsto& 
(1, y) ~~\text{if $y + y < x$}\\
&&&
(1, y + \underline{1}) ~~ \text{otherwise} & & & &
(1, y + \underline{1})~~ \text{otherwise}\\
\end{array}
\]
\begin{figure}
\[\begin{tikzcd}[cramped]
	{} & {\phantom{\vdots}} && {} & {\phantom{\vdots}} && {} & {\phantom{\vdots}} &&&&&& {\phantom{\vdots}} &&& {\phantom{\vdots}} &&& {\phantom{\vdots}} \\
	\vdots & \vdots && \vdots & \vdots && \vdots & \vdots &&&&&& {\phantom{\vdots}} & {\phantom{\mathbb{N}_\infty}} && {\phantom{\vdots}} & {\phantom{\mathbb{N}_\infty}} && {\phantom{\vdots}} & {\phantom{\mathbb{N}_\infty}} \\
	{(1,\underline{n+1})} & {(1,\underline{n+1})} && {(1,\underline{n+1})} & {(1,\underline{n+1})} && {(1,\underline{n+1})} & {(1,\underline{n+1})} \\
	{(1,\underline{n})} & {(1, \underline{n})} && {(1,\underline{n})} & {(1, \underline{n})} && {(1,\underline{n})} & {(1, \underline{n})} \\
	\vdots & \vdots && \vdots & \vdots && \vdots & \vdots &&&&&&& {\phantom{\mathbb{N}_\infty}} &&& {\phantom{\mathbb{N}_\infty}} &&& {\phantom{\mathbb{N}_\infty}} \\
	& {} &&& {} &&& {} \\
	{} & {} && {} & {} && {} & {}
	\arrow[dashed, from=1-1, to=2-2]
	\arrow[draw=none, from=1-2, to=2-2]
	\arrow[draw=none, from=1-5, to=2-5]
	\arrow[draw=none, from=1-8, to=2-8]
	\arrow[draw=none, from=1-14, to=2-14]
	\arrow[draw=none, from=1-17, to=2-17]
	\arrow[draw=none, from=1-20, to=2-20]
	\arrow[dashed, from=2-1, to=3-2]
	\arrow[dashed, from=2-2, to=2-1]
	\arrow[dashed, from=2-4, to=2-5]
	\arrow[dashed, from=2-5, to=1-4]
	\arrow[dashed, from=2-7, to=2-8]
	\arrow[dashed, from=2-8, to=1-7]
	\arrow[draw=none, from=2-15, to=5-15]
	\arrow[draw=none, from=2-18, to=5-18]
	\arrow[draw=none, from=2-21, to=5-21]
	\arrow[from=3-1, to=4-2]
	\arrow[from=3-2, to=3-1]
	\arrow[from=3-4, to=3-5]
	\arrow[dashed, from=3-5, to=2-4]
	\arrow[from=3-7, to=4-8]
	\arrow[dashed, from=3-8, to=2-7]
	\arrow[draw=none, from=4-1, to=3-1]
	\arrow[dashed, from=4-1, to=5-2]
	\arrow[from=4-2, to=4-1]
	\arrow[draw=none, from=4-4, to=3-4]
	\arrow[dashed, from=4-4, to=5-5]
	\arrow[from=4-5, to=4-4]
	\arrow[draw=none, from=4-7, to=3-7]
	\arrow[dashed, from=4-7, to=5-8]
	\arrow[from=4-8, to=4-7]
	\arrow[dashed, from=5-1, to=6-2]
	\arrow[dashed, from=5-2, to=5-1]
	\arrow[dashed, from=5-4, to=6-5]
	\arrow[dashed, from=5-5, to=5-4]
	\arrow[dashed, from=5-7, to=6-8]
	\arrow[dashed, from=5-8, to=5-7]
	\arrow["{f_\infty, g_\infty}"{description}, draw=none, from=7-1, to=7-2]
	\arrow["{f_{\underline{2n}}, g_{\underline{2n}}}"{description}, draw=none, from=7-4, to=7-5]
	\arrow["{f_{\underline{2n+1}}, g_{\underline{2n+1}}}"{description}, draw=none, from=7-7, to=7-8]
\end{tikzcd}\]
\caption{Picturing graphs spanned by $(f_x, g_x)$. If $x = \infty$, we have a single connected component,
otherwise if $x = \underline{m}$ we have a break around $(1, \underline{\left\lfloor\frac{m}{2}\right\rfloor})$.}
\label{fig:brokenLadder}
\end{figure}
Since canonical codes for $f_x$ and $g_x$ can be uniformly computed from $x$, our assumption
tells us that there exist continuous maps
\[H, H_{\mathrm{inv}} ~~:~~ \Ninfty \times (2 \times \Ninfty) \longto 2 \times \Ninfty\]
such that $H(x, -)$ and $H_{\mathrm{inv}}(x,-)$ are mutually inverse bijections and such that
\[\forall (i, y) \in 2 \times \Ninfty. ~ \neg\neg \exists m \in \bZ. \; H(x,(i,y)) = \left(f_x \circ (g_x \circ f_x)^m\right)(i,y) \]
Let us show that this is absurd. First, observe that the following map is continuous and has a discrete codomain.
\[
\begin{array}{llcl}
H_{\mathrm{ground}} :& \Ninfty \times \Ninfty &\longto& 2 \times 2\\
& (x, y) & \longmapsto& (\pi_1(H(x,(1,y))),\pi_1(H_{\mathrm{inv}}(x,(1,y))))\\
\end{array}
\]
So in particular, there is some $n \in \bN$ such that $H_{\mathrm{ground}}(x, y)$ always takes the same value
for any $x, y \in \Ninfty$ above $\underline{n}$. Call $A_n$ the finite set $\{1\} \times \{\underline{0}, \ldots, \underline{n-1}\} \subseteq 2 \times \Ninfty$. Now we want to show that there is some $m > n$ such that
\[ \forall z \in A_n. \; \forall y \ge \underline{m}. ~~
\sup(\pi_2(H(\infty, z)),\pi_2(H_{\mathrm{inv}}(\infty, z))) \le \underline{m}
\quad \text{and} \quad
\left[\begin{array}{lcl}
H(\infty,z) &=& H(y,z)\\
H_{\mathrm{inv}}(\infty,z) &=& H_{\mathrm{inv}}(y,z)
\end{array}\right.
\]
Since $A_n$ is finite, it is enough to find a bound for each individual $z \in A_n$ and then take a maximum; we can also treat $H$ and $H_{\mathrm{inv}}$ separately that way. Also for each $z \in A_n$, we already know that $\pi_2(H(\infty, z)) \in \underline{\bN}$ since $H(\infty, -)$ must be compatible with $(f_\infty,g_\infty)$,
so bounding that is no issue either.
So we can focus on finding $m_z$ such that $H(\infty,z) = H(y,z)$ for all $y \ge \underline{m_z}$.
Since $\{z\}$ and $\{H(\infty,z)\}$  are clopens in $2 \times \Ninfty$,
the preimage of $H(\infty, z)$ intersected with $\Ninfty \times \{z\}$ is a clopen which contains $(\infty, z)$.
But then it means that the first component is necessarily cofinite; any witness of cofiniteness is a suitable $m_z$.

Now, let us consider the following two sets:
\[
S_L = H(\infty,A_n) ~~\cap~~ \{1\} \times \{\underline{n}, \ldots, \underline{m}\}
\qquad\text{and}\qquad
S_R = H_{\mathrm{inv}}(\infty,A_n) ~~\cap~~ \{1\} \times \{\underline{n}, \ldots, \underline{m}\}
\]
We now make a case distinction according to whether $|S_L| = |S_R|$ or not:
\begin{itemize}
\item if $|S_L| \neq |S_R|$, we run into a contradiction when checking that $H(\underline{2m},-)$ is
a suitable bijection (stage 2a in~\Cref{fig:noNNinftyMP}): at this stage $H$ and $H_{\mathrm{inv}}$
must produce outputs whose first component is $1$ and should thus induce a bijection between
the complements of $S_L$ and $S_R$ in $\{\underline{n}, \ldots, \underline{m}\}$, which is impossible
since they have different size $|S_L| - m + n$ and $|S_R| - m + n$.
\item otherwise, we run into a contradiction when checking that $H(\underline{2m+1},-)$ works: if it did,
it would induce a bijection between $\{\underline{n}, \ldots, \underline{m}, \underline{m+1}\} \setminus S_L$
and $\{\underline{n}, \ldots, \underline{m}\} \setminus S_R$, which is impossible since their sizes differ by one.
\end{itemize}
\end{proof}

Now that we know it is impossible to uniformly compute a suitable bijection, we can
still ask if we can sharpen~\Cref{thm:contbij2Ninfty} to gauge how non-computable this
problem is.
Here we sketch how $\dneg$ could be replaced with $\modal_{\LPO^\diamond}$
in the statement of~\Cref{thm:contbij2Ninfty}\footnote{
Recalling that modalities can essentially be defined for any Weihrauch problem as per~\Cref{rem:modWei},
with a bit more work, one can replace $\LPO^\diamond$ with $\LLPO^* \star \LPO^8$ in~\Cref{conj:2NinftyModLPOdia};
if we further assumed $f$ and $g$ to be continuous, $\LLPO^*$ would suffice.}.

\begin{claim}
  \label{conj:2NinftyModLPOdia}
Assuming $\MP$, given two injections $f, g : 2 \times \Ninfty \to 2 \times \Ninfty$,
there $\modal_{\LPO^\diamond}$-exists a bijection $h : 2 \times \Ninfty \to 2 \times \Ninfty$
such that
\[\forall x \in 2 \times \Ninfty. \;
\modal_{\LPO} \left(h(x) \in \bigcup_{m \in \bZ} \left(f \circ (g \circ f)^m\right)(x)\right)\]
\end{claim}
\begin{proof}[Proof sketch]
One may compute a ``modulus of (uniform) continuity'' as an element of $\Ninfty$
for $\pi_1 \circ f$ and $\pi_1 \circ g$. Using $\modal_\LPO$, we can reason
according to whether it is $\infty$, in which case it is clear that either
$f$ or $g$ is discontinuous and we can use~\Cref{lem:oscillation} to deduce
$\LPO$, or get a large enough $n \in \bN$ as per~\Cref{lem:continjNNinfty}.
We can similarly use a couple of $\LPO$ questions to determine whether $f$ and
$g$ send non-isolated points to non-isolated points, and deduce $\LPO$ if that is not
the case. Once we have the confirmation that $f$ and $g$ are continuous and
that $n$ is a modulus of continuity for $\pi_1 \circ f$ and $\pi_1 \circ g$,
we then put ourselves without loss of generality in the case where
both $f$ and $g$ are untangled ($f(0,\infty) = g(0, \infty) = (0, \infty)$)
by using a similar trick as in the proof of~\Cref{thm:contbij2Ninfty}.
We can then start a back-and-forth to determine $h$ and $h^{-1}$ on
$2 \times \{0, \ldots, n-1\}$. Once that is done, note that the selection
function for $\Ninfty$ allows us to compute, for any map
$\iota : 2 \times \Ninfty \to 2 \times \Ninfty$ and $(i, k) \in 2 \times \bN$
a potential element in $\iota^{-1}(i, \underline{k}) \cap \{i\} \times \Ninfty$
(i.e. the element is in $\iota^{1}(i, \underline{k}) \cap \{i\} \times \Ninfty$
if and only if that set is inhabited at all).
Then, for any
$(i, k) \in 2 \times \bN$, we can thus compute a sequence $p_{i,k}$ such that, for any $m \in \bN$,
\[ p_{i,k}(m) \in (f \circ g)^{-1}(i,k) \quad \Longleftrightarrow
  \quad \forall m' \le m. \; \exists x \in \Ninfty. \; 
(i, x) \in (f \circ g)^{-m'}(i,k)\]
Using one $\LPO$ question, one can check whether
$f(g(i,p_{i,k}(m+1))) = (i, p_{i, k}(m))$ and $p_{i,k}(m) \ge \underline{n}$
for every $m \in \bZ$. So we can detect
whether it is the case that some $(i, \underline{k})$ with $k < n$ has
only a finite and odd\footnote{Because the problematic cases are restricted
to those odd number of predecessors, an $\LLPO$ question instead of the $\LPO$ question
would suffice.} number of predecessors that do not sit in a cycle outside of the
danger zone that $2 \times \{0, \ldots, n-1\}$ is. In such a case, extend the
back-and-forth construction so that those finitely many predecessors get assigned
a value by $h$ (and do a similar process for $h^{-1}$). Since the danger zone is finite, the back-and-forth does not
determine $h$ and $h^{-1}$ on any values $(i,x)$ for $x \ge \underline{n'}$ for some $n' > n$.
For any such value, set $\pi_1(h(i,x)) = \pi_1(h^{-1}(i,x)) = i$.
The second components of $h$ and $h^{-1}$ are then
determined by a back-and-forth that respects this constraint. The back-and-forth
is still carried out as in the proof of Myhill's isomorphism theorem, save for
a special case: if we are ever led by the usual procedure to take a path
entering the danger zone to determine $h(i,x)$ (or $h^{-1}(i,x)$) for $x \ge \underline{n'}$, then
we know that $(i,x)$ has infinitely many predecessors outside of the danger
zone, so we rather attempt to travel the graph in reverse direction. For that,
we use that predecessors are computable using the selection function for $\Ninfty$.
Reasoning that $h$ and $h^{-1}$ are then well-defined can be done in a similar
way as in the proof of~\Cref{thm:ninftymyhill}.

Finally, to check that for any $(i, x) \in 2 \times \Ninfty$ we do have
\[\modal_{\LPO} \left(h(i,x) \in \bigcup_{m \in \bZ} \left(f \circ (g \circ f)^m\right)(i,x)\right)\]
the only non-trivial case is if we are in the final case where $f$ and $g$ are
continuous (and we do not have directly unrestricted access to $\LPO$). In such
a case, we use the modality to determine whether $x = \infty$ or $x = \underline{k}$
for some $k \in \bN$. The first case is easy as the construction guarantees
$h(i,\infty) = f(i, \infty)$, in the second case we need to inspect the back-and-forth
construction after it has assigned values for $h \cup h^{-1}$ for all pairs $(i, \underline{k'})$
for $k' \le \max(k,n')$.
\end{proof}

\subsection{Parameterized counter-examples}

Here we show that the $\neg\neg$ Myhill property does not play well with
(co)products and exponentials. The
main idea is to substitute $\Ninfty$ for $\bN$ in the arguments of~\Cref{subsec:negEM}
and make the obvious adaptations. The main point is that reliance on $\LPO$ is avoided since
there is a selection function on $\Ninfty$, but we sometimes need to assume $\MP$
to have that $\Ninfty$ has the $\dneg$ Myhill property.

\begin{lemma}
For any $p \in \Omega$, if $(\{0,1\} \uplus \{2 \mid p\})^\bN$ has the $\dneg$ Myhill property,
then $p \vee \neg p$ holds.
\end{lemma}
\begin{proof}
Same as~\Cref{lem:expEM} modulo the easy check that the subsets in play are $\dneg$-stable.
\end{proof}

\begin{lemma}
For any $p \in \Omega$, if $p + \Ninfty$ has the $\dneg$ Myhill property,
then $p \vee \neg p$ holds.
\end{lemma}
\begin{proof}
Take $A = \{ \inr(x + x) \mid x \in \Ninfty\}$ and $B = A \cup \{\inl(0) \mid p\}$.

Then we have injections 
$f$ and $g$ that witness $A \preceq_1^{p + \Ninfty} B$ and $B \preceq_1^{p + \Ninfty} A$
respectively:
\[
\begin{array}{llcl !\qquad lcl}
& f(\inr(z)) &=& \inr(\underline{2} + z) &
g(\inr(z)) &=& \inr(\underline{2} + z) \\
 \text{and, } &
f(\inl(0)) &=& \inr(\underline{1}) &
g(\inl(0)) &=& \inr(\underline{0})
\end{array}
\]

Now, since $p + \Ninfty$ has the $\dneg$ Myhill property, we have a bijection $h : p + \Ninfty \to p + \Ninfty$
such that $h(A) = B$. By \Cref{thm:Ninftysel}, we can decide whether
$\inl(0)$ belongs to the image of $h$. If it does, then we know that $p$ holds.
Otherwise, assuming $p$ leads to $\inl(0) \in B = h(A)$ and a contradiction;
thus $\neg p$ holds.
\end{proof}

\begin{lemma}
For any $p \in \Omega$, if $(p \uplus \{1\}) \times \Ninfty$ has the $\dneg$ Myhill property,
then $p \vee \neg p$ holds.
\end{lemma}
\begin{proof}
Take $A = \{ (1, \inr(x + x)) \mid x \in \Ninfty\}$ and $B = A \cup \{(0, \underline{0}) \mid p\}$.
Then we have injections 
$f$ and $g$ that witness $A \preceq_1^{(p \uplus \{1\}) \times \Ninfty} B$ and $B \preceq_1^{(p \uplus \{1\}) \times \Ninfty} A$
respectively:
\[
\begin{array}{llcl !\qquad lcl}
& f(1, z) &=& (1,\underline{2} + z) &
g(1,z) &=& (1, \underline{2} + z) \\
 \text{and,} &
 f(0, x) &=& (0, x + \underline{1}) &
 g(0, \underline{0}) &=& (1, \underline{0}) \\
& & & & g(0, x + \underline{1}) &=& (0, x) \\
\end{array}
\]
Now, if $p + \Ninfty$ has the Myhill property, then we have a bijection $h :
(p \uplus \{1\}) \times \Ninfty \to (p \uplus \{1\}) \times \Ninfty$
with $h(A) = B$. By \Cref{thm:Ninftysel}, we can decide whether there exists
$x \in \Ninfty$ such that $h(1,x) = (0, \underline{0})$ or not.
If there is, then we know that $p$ must hold.
Otherwise, assuming $p$, we have $(0, \underline{0}) \in B$.
But since $B = h(A)$ and $A \subseteq \{1\} \times \Ninfty$, we then would
have a contradiction. Hence $\neg p$ holds in that case.
\end{proof}

\begin{corollary}
Markov's principle and ``if $A$ and $B$ have the $\dneg$ strong Myhill property entails
that either $A + B$ or $A \times B$ has the $\dneg$ Myhill property'' imply excluded middle.
\end{corollary}
\begin{proof}
Put together the previous lemmas with \Cref{lem:subfinite} and \Cref{thm:ninftymyhill}.
\end{proof}

\section{Conclusion}

The overall picture is rather bleak if one is hoping to naively and constructively
extend the Myhill isomorphism theorem to sets other than $\bN$ or subfinite
sets. While it is possible to get explicit
constructions for $\Ninfty$ (at the cost of inserting some double-negations in
the correctness condition and accepting Markov's principle), a number of natural examples including
$2^\bN$ and $\Baire$ fail to have the same property.

However, something that could be fun would be to attempt to calibrate precisely how
hard this failure is in the style of constructive reverse mathematics.
For instance, for $\bN + \Ninfty$, $\bN \times \Ninfty$ and $\Ninfty^2$, we
know that those having the $\dneg$ Myhill property is equivalent to $\LPO$.
We could ask a similar question for $2^\bN$ and $\Baire$.
Aside from equivalences, another way one could calibrate is by using
modalities in front of the existential statements (as in~\Cref{conj:2NinftyModLPOdia}),
and see what are the least powerful modalities that do the job.

One could also use the generalization of modalities to
act on the universe of all sets instead of propositions as discussed
in~\cite{RSSmod,Swan24}\footnote{A similar use of endofunctors also
occur in computable analysis~\cite{PaulydBrechtSDST} to e.g. define representations
for certain classes of non-continuous maps.}.
In this framework, given a set $X$, an element of
$a \in \modal_{\LPO} X$ should be regarded as a constant partial function that yields
an element of $X$ if it is given as input a solution to some (uniquely determined)
$\LPO$ question. Then, another weakening of the $\dneg$ Myhill property could
ask, in lieu of a bijection $X \to X$, a map
$h : X \to \modal X$ and an ``inverse'' $h^{-1} : X \to \modal X$.

This sort of fine calibration might also be more interesting to trial on theorems of
a similar flavor, such as strong variants of the Cantor-Bernstein theorem~\cite{Banaschewski-1986}
or ``set division problems'' whose (non)-constructivity has received
recent attention~\cite{doyle2015division,S18dividing,MO22}.

\printthanks

\bibliographystyle{eptcs}
\bibliography{biblio}

\end{document}